\documentclass[12pt]{amsart}
\usepackage{txfonts}
\usepackage{amssymb}
\usepackage{eucal}
\usepackage{graphicx}
\usepackage{amsmath}
\usepackage{amscd}
\usepackage[all]{xy}           

\usepackage{amsfonts,latexsym}
\usepackage{xspace}
\usepackage{epsfig}
\usepackage{float}
\usepackage{color}
\usepackage{fancybox}
\usepackage{colordvi}
\usepackage{multicol}
\usepackage{colordvi}
\usepackage{stmaryrd}

\newtheorem{theorem}{Theorem}[section]

\newtheorem{lemma}[theorem]{Lemma}
\newtheorem{coro}[theorem]{Corollary}
\newtheorem{prop-def}{Proposition-Definition}[section]

\newtheorem{exam}[theorem]{Example}

\newcommand{\nc}{\newcommand}
\newcommand{\delete}[1]{}

\nc{\mlabel}[1]{\label{#1}}  
\nc{\mcite}[1]{\cite{#1}}  
\nc{\mref}[1]{\ref{#1}}  
\nc{\mbibitem}[1]{\bibitem{#1}} 

\delete{
\nc{\mlabel}[1]{\label{#1}  
{\hfill \hspace{1cm}{\bf{{\ }\hfill(#1)}}}}
\nc{\mcite}[1]{\cite{#1}{{\bf{{\ }(#1)}}}}  
\nc{\mref}[1]{\ref{#1}{{\bf{{\ }(#1)}}}}  
\nc{\mbibitem}[1]{\bibitem[\bf #1]{#1}} 
}

\nc{\bfk}{\mathbf{k}}
\nc{\Der}{\mathrm{Der}}
\nc{\Ker}{\mathrm{Ker}}

\begin{document}

\title{ Manin triples of 3-Lie algebras induced by involutive derivations }

\author{Shuai Hou}
\address{College of Mathematics and Information  Science,
Hebei University, Baoding 071002, China} \email{hshuaisun@163.com}

\author{RuiPu  Bai}
\address{College of Mathematics and Information Science,
Hebei University
\\
Key Laboratory of Machine Learning and Computational\\ Intelligence of Hebei Province, Baoding 071002, P.R. China} \email{bairuipu@hbu.edu.cn}
\footnotetext{ Corresponding author: Ruipu Bai, E-mail: bairuipu@hbu.edu.cn.}
\date{}


\begin{abstract}

For any $n$-dimensional 3-Lie algebra $A$ over a field of characteristic zero with an involutive derivation $D$, we investigate the structure  of the 3-Lie algebra $B_1=A\ltimes_{ad^*} A^* $  associated with the coadjoint representation $(A^*, ad^*)$. We then discuss the structure of the dual 3-Lie algebra $B_2$ of the local cocycle 3-Lie bialgebra $(A\ltimes_{ad^*} A^*, \Delta)$. By means of  the involutive derivation $D$, we construct the $4n$-dimensional Manin triple $(B_1\oplus B_2,$ $  [ \cdot, \cdot, \cdot]_1,$ $  [ \cdot, \cdot, \cdot]_2,$ $ B_1, B_2)$ of 3-Lie algebras, and provide concrete multiplication in a special basis $\Pi_1\cup\Pi_2$. We also construct a sixteen dimensional  Manin triple $(B, [ \cdot, \cdot, \cdot])$ with  $\dim B^1=12$ using an involutive derivation on a four dimensional 3-Lie algebra $A$ with $\dim A^1=2$.

\end{abstract}

\subjclass[2010]{17B05, 17D99.}

\keywords{ 3-Lie algebras, involutive derivations, semi-direct product 3-Lie algebra, Manin triples.}

\maketitle



\allowdisplaybreaks

\section{Introduction}

The notion of n-Lie algebras was introduced by Filippov in \cite{FV}, which are closely related to the fields of mathematics and physics, and the algebraic structure of n-Lie algebras corresponds  to Nambu mechanics \cite{Gautheron P, Dirac, Takhtajan, Awata}. In particular,  as a special case of n-Lie algebras, 3-Lie algebras are extensively   studied  because they play a significant role in string theory and M-theory \cite{Bagger, Gustavsson, Bandos, Sheikh-Jabbari,Gauntlett}. For example, the basic model of Bagger-Lambert-Gustavsson theory is based on the structure of metric 3-Lie algebras, and the
Jacobi equation of 3-Lie algebras is the foundation for defining the $N=8$ supersymmetry action.

 Lie bialgebras have widespread applications in geometry and physics. The structure of Lie bialgebras is very important since it contains coboundary theory, which makes the structure of Lie bialgebras  relate to the classical Yang-Baxter equation \cite{Gerstenhaber}.  In general,
a Lie bialgebra is actually  a vector space  endowed with  a Lie algebra structure $(A,[\cdot,\cdot])$ and a Lie coalgebra structure $(A, \Delta )$ (where $\Delta: A\rightarrow \wedge^2 A$ is the comultiplication)  satisfying  the  compatibility condition which is proposed based on the  Hamitonian dynamics and Poisson Lie groups \cite{Drinfeld, Etingof, Loday, Joni}. In \cite{BaiGL, BCM5}, the authors studied 3-Lie bialgebra structures and  Manin-triple for 3-Lie algebras.

In this paper, we mainly study the structure of 3-Lie algebras with involutive derivations. By means of an involutive derivation of an $n$-dimensional 3-Lie algebra $A$, we construct a $4n$-dimensional Manin triple $(B_1\oplus B_2,$ $  [ \cdot, \cdot, \cdot]_1,$ $  [ \cdot, \cdot, \cdot]_2,$ $ B_1, B_2)$ of 3-Lie algebras, and study its structure.  We also construct a 16-dimensional  Manin triple $(B, [ \cdot, \cdot, \cdot])$ of 3-Lie algebras by  means of  an involutive derivation $D$ on a four dimensional 3-Lie algebra $A$ with $\dim A^1=2$, which  is a 16-dimensional 3-Lie algebra with $\dim B^1=12.$

The paper is organized as follows. In section 2, we recall some elementary facts on 3-Lie algebras,  and give the description of
     the  semi-direct product 3-Lie algebras associated with the coadjoint representations of  3-Lie algebras which have  involutive derivations. In section 3, by means of involutive derivations, we construct a class of local cocycle $3$-Lie bialgebras. In section 4, based on the
coadjoint representations of semi-direct product 3-Lie algebras and the dual structures of the 3-Lie coalgebras, we construct a class of Manin triples and Matched pairs of 3-Lie algebras.

In the  paper, we suppose that all algebras and vector spaces are over a field $\mathbb F$ of characteristic zero, and for a subset $S$ of a vector space $V$, we use $\langle S\rangle$ to denote the
subspace of $V$ spanned by $S$.

\section{The semi-direct product 3-Lie algebra $A\ltimes_{ad^*}A^*$}

 First we recall the notion of 3-Lie algebras with involutive derivations.

 {\it A 3-Lie algebra} is a vector space $A$  with a  linear  multiplication (or 3-Lie bracket)
$[\cdot,\cdot,\cdot]:\wedge^{3}A\rightarrow A$ satisfying
\begin{equation}
  [x_1,x_2,[x_3,x_4, x_5]]=[[x_1,x_2, x_3],x_4,x_5]+[x_3,[x_1,x_2, x_4],x_5]+[x_3,x_4,[x_1,x_2, x_5]],
 \label{eq:jacobi1}
\end{equation}
for $\forall x_{i}\in A, 1\leq i\leq 5$.

{\it A derivation} $D$ of a 3-Lie algebra $A$ is  a linear map $D:A\rightarrow A$  satisfying,
\begin{equation}
 D([x_1,x_2,x_3])=[D(x_1),x_2,x_3]+[x_1,D(x_2),x_3]+[x_1,x_2,D(x_3)],\forall x_1,x_2,x_3\in A.
\label{eq:der}
\end{equation}
In addition, if $D^2=I_d$, then $D$ is called an {\it involutive derivation} on $A$.

Thanks to \eqref{eq:jacobi1}, for $\forall x_1, x_2\in A,$ the left multiplication
$$ad_{x_1x_2}:\wedge^{2}A\rightarrow gl(A)$$ defined by
\begin{equation}\label{eq:ad}
ad_{x_{1}x_{2}}x=[x_{1},x_{2},x], ~~~ \forall x\in A,
\end{equation}
satisfies
\begin{equation}
ad_{x_{1}x_{2}}[x_{3},x_{4},x_{5}]=[ad_{x_{1}x_{2}}x_{3},x_{4},x_{5}]+[x_{3},ad_{x_{1}x_{2}}x_{4},x_{5}]+[x_{3},x_{4},ad_{x_{1}x_{2}}x_{5}],
\label{eq:ad2}
\end{equation}
which is called {\it  an inner derivation}.

Let $A$ be an $n$-dimensional 3-Lie algebra with an involutive derivation $D$. Then $A$ has  a decomposition
\begin{equation}
    A=A_1 ~\dot+~ A_{-1},
    \label{eq:1}
\end{equation}
where  $A_1=\{x\in A \mid Dx=x\}$ and $ A_{-1}=\{
x\in A \mid Dx=-x\}$, and there is a basis $\{x_1, \cdots, x_n\}$ of $A$ such that
$x_1,\cdots,x_s\in A_{1},$ and $x_{s+1},\cdots,x_n\in A_{-1}$.

\begin{lemma} \label{lem:bb}
Let $A$ be a finite dimensional 3-Lie algebra with an
involutive derivation $D$.
 Then
\begin{equation}
\hspace{15mm}[A_1, A_1, A_1]=[A_{-1}, A_{-1}, A_{-1}]=0,
\end{equation}

\begin{equation}
[A_1, A_1, A_{-1}]\subseteq A_1, \quad [A_1, A_{-1}, A_{-1}]\subseteq A_{-1}.
\end{equation}
\end{lemma}

\begin{proof} Apply Theorem 4 in \cite{Derivation}.
\end{proof}

 {\it A representation (or an  $A$-module)}\cite{KA} of a 3-Lie algebra $A$ over a field $\mathbb F$   is a pair $(V, \rho)$, where $V$ is a vector space over  $\mathbb F$, and $\rho$ is an $\mathbb F$-linear  map $\rho:\wedge^{2}A \rightarrow gl(V)$ satisfying,  for $\forall x_{1}, x_{2}, x_{3}, x_{4}\in A,$
  \begin{equation}\label{eq:i}
  [\rho(x_{1},x_{2}),\rho(x_{3},x_{4})]=\rho([x_{1},x_{2},x_{3}],x_{4})+\rho(x_{3},[x_{1},x_{2},x_{4}]),
  \end{equation}
  \begin{equation}
  \label{eq:ii}\rho([x_{1},x_{2},x_{3}],x_{4})=\rho (x_{1},x_{2})\rho (x_{3},x_{4})+\rho (x_{2},x_{3})\rho (x_{1},x_{4})+\rho (x_{3},x_{1})\rho (x_{2},x_{4}).
  \end{equation}

 There is an equivalent description for an $A$-module, that is, $(V, \rho)$ is an $A$-module if and only if
$(A\oplus V, \mu)$ is a 3-Lie algebra, where $\mu: (A\oplus V)^3\rightarrow A\oplus V$, for all $x_1, x_2, x_3\in A, v\in V$,
$$\mu(x_1, x_2, x_3)=[x_1, x_2, x_3], ~~ \mu(x_1, x_2, v)=\rho(x_1, x_2)v, ~~ [A, V, V]=[V, V, V]=0.$$
The 3-Lie algebra $(A\oplus V, \mu)$ is called {\it the semi-direct product 3-Lie algebra of $A$ associated with $(V, \rho),$} which is denoted by $A\ltimes_{\rho}V. $ \cite{RE}

Let $(A,[\cdot,\cdot,\cdot])$ be a 3-Lie algebra. Thanks to \eqref{eq:ad} and \eqref{eq:ad2}, $(A, ad)$ is a representation, where $ad:\wedge^{2} A\rightarrow gl(A)$, for all $x_1, x_2\in A$, $ad(x_1\wedge x_2)=ad_{x_1x_2}$, which is called the {\it adjoint representation} of $A$. The dual representation $(A^{*}, ad^{*})$  of $(A, ad)$ is called the {\it coadjoint representation}, where  $ad^*: \wedge^2A\rightarrow gl(A^*)$ is defined by
\begin{equation}\label{eq:one}
 \langle ad^{*}_{x_1x_2}x_{c}^*, x_{t}\rangle=-\langle x_{c}^*,ad_{x_1x_2}x_{t}\rangle, ~ \forall x_1,x_2,x_{t}\in A, x_{c}^*\in A^*.
 \end{equation}

The semi-direct product 3-Lie algebra  $(A\ltimes _{ad^*} A^*,\mu)$ associated with $(A^*, ad^*)$ is denoted by $B_1.$  Then
 for all $x_i\in A, x_{i}^{*}\in A^{*}, 1\leq i \leq 3$,
\begin{equation}\label{eq:*}
\small\mu(x_1+x_{1}^{*},x_2+x_{2}^{*},x_3+x_{3}^{*})=[x_1,x_2,x_3]+ad^{*}_{x_1x_2}x_{3}^{*}+ad^{*}_{x_3x_1}x_{2}^{*}+ad^{*}_{x_2x_3}x_{1}^{*}.
\end{equation}

For describing the structure of the semi-direct product 3-Lie algebra  $(A\ltimes _{ad^*} A^*,\mu)$, we need the structural constant in a basis of the 3-Lie algebra $A$.

Let $A$ be a 3-Lie algebra with a basis $\{x_1,\cdots,x_n\}, $ and let $ \{{x_1^{*},\cdots,x_n^{*}}\}$ be the dual basis of the dual space $A^{*},$ that is,
\begin{equation}\label{eq:sasa}
\langle x_i,x^{*}_j\rangle=\delta_{ij}, 1\leq i,j\leq n.
\end{equation}
Suppose the multiplication of $A$ in the basis  \{$x_1,\cdots,x_n$\} is
\begin{equation}\label{eq:ooo}
[x_a,x_b,x_c]=\sum_{k=1}^n \Gamma ^{k}_{abc}x_{k}, \quad \Gamma ^{k}_{abc}\in \mathbb F,\quad  1\leq a,b,c,k\leq n.
\end{equation}

 By the above notations, we have the following result.

\begin{theorem}\label{thm:semimu} Let $A$ be a 3-Lie algebra with an involutive derivation $D$, and let the multiplication of $A$ in the basis $\{x_1,\cdots,x_n\}$ be \eqref{eq:ooo}, where
 $$A_1=\langle x_1, \cdots, x_s\rangle, \quad A_{-1}=\langle x_{s+1}, \cdots, x_n\rangle.$$
Let furtherd  $\{x_1^{*},\cdots,x_n^{*}\}$ be the dual basis of the dual space $A^*$. Then  the mulitiplication $\mu$ of the semi-direct product 3-Lie algebra  $B_{1}=A\ltimes_{ad^*} A^*$ satisfies
\begin{equation}\label{eq:1}
\mu(x_a,x_b,x_c)=\begin{cases}\begin{split}
\sum_{k=1}^s \Gamma ^{k}_{abc}x_{k}, &1\leq a,b\leq s<c\leq n,\\
\sum_{k=s+1}^n \Gamma ^{k}_{abc}x_{k}, &1\leq a\leq s<b,c\leq n,\\
0,~ &~~1\leq a,b,c\leq s ~~or~ s+1\leq a,b,c\leq n.\\
\end{split}
\end{cases}
\end{equation}

\begin{equation}\label{eq:2}
\hspace{-1.2cm}\mu(x_a,x_b,x_c^{*})=
\begin{cases}\begin{split}
-\sum_{k=s+1}^n \Gamma ^{c}_{abk}x_{k}^{*},&~~ 1\leq a,b,c\leq s,\\
-\sum_{k=1}^s \Gamma ^{c}_{abk}x_{k}^{*}, &~~  s+1\leq a,b,c\leq n,\\
-\sum_{k=1}^s \Gamma ^{c}_{abk}x_{k}^{*},&~~   1\leq a,c\leq s<b\leq n,\\
-\sum_{k=s+1}^n \Gamma ^{c}_{abk}x_{k}^{*},&~~   1\leq a\leq s<b,c\leq n,\\
0, &~~  1\leq a,b\leq s<c\leq n,\\
0,&~ 1\leq c\leq s<a,b\leq n,\\
\end{split}\end{cases}
\end{equation}
and $$\mu(x_a, x_b^*, x_c^{*})=\mu(x_a^*,x_b^*,x_c^{*})=0, \quad \forall 1\leq a, b, c\leq n.$$
\end{theorem}

\begin{proof}  From \eqref{eq:*}, we can suppose
\begin{equation}\label{eq:ppp}
\mu(x_{a},x_{b}, x_{c}^{*})=ad^{*}_{x_{a}x_{b}}x_{c}^{*}=\sum_{k=1}^n\lambda_{abc}^{k}x_{k}^{*}, \quad \lambda_{abc}^{k}\in F,\quad  1\leq a,b,c,k\leq n.
\end{equation}

Thanks to  Lemma \ref{lem:bb} and Eq \eqref{eq:ooo}, a direct computation yields \eqref{eq:1}.

 By Eqs \eqref{eq:one} and \eqref{eq:ppp}, for $\forall x_{a}, x_{b}, x_{t}\in A, x_{c}^{*}\in A^{*},$

\begin{equation*}\begin{split}
\langle\mu(x_{a},x_{b}, x_{c}^{*}),x_{t}\rangle=&\langle\sum_{k=1}^n \lambda^{k}_{abc}x_{k}^{*},x_{t}\rangle=\lambda_{abc}^{t},\\
\langle\mu(x_{a},x_{b}, x_{c}^{*}),x_{t}\rangle =&-\langle x_{c}^{*},ad_{x_{a},x_{b}}x_{t}\rangle=-\langle x_{c}^{*},\sum_{k=1}^n \Gamma^{k}_{abt}x_{k}\rangle=-\Gamma^{c}_{abt}.
\end{split}
\end{equation*}
Therefore,$$ \lambda_{abc}^{t}=-\Gamma^{c}_{abt}, \quad \forall 1\leq a, b, c, t\leq   n.$$
By  \eqref{eq:ppp}, $$\mu(x_a,x_b,x_c^{*})=\sum_{k=1}^n\lambda_{abc}^{k}x_{k}^{*}=-\sum_{k=1}^n\Gamma_{abk}^{c}x_{k}^{*}, \quad \forall 1\leq a, b, c, t\leq   n.$$
Follows from \eqref{eq:1},  we obtain \eqref{eq:2}. The proof is complete.
\end{proof}

\section{The local cocycle 3-Lie bialgebras induced by  involutive derivations}

For a 3-Lie algebra $A$ with an involutive derivation $D$, we will construct a 3-Lie algebra $B_2$ on the dual space $(A\oplus A^*)^{*}$.
But before that, we need to construct a local cocycle 3-Lie bialgebra structure on the semi-direct product 3-Lie algebra $A\ltimes_{ad^*}A^*$.
First we recall some definitions.

Let $A$ be a 3-Lie algebra, $A^*$ be the dual space of $A$, and $\Delta: A\rightarrow A\wedge A\wedge A$ be a linear mapping. Then  the dual mapping $\Delta^*$ of $\Delta$ is a linear mapping
 $\Delta^*: A^*\wedge A^*\wedge A^*\rightarrow
A^*$ satisfying,  
\begin{equation}\label{eq:dualdelta}
 \langle \Delta^*(\alpha, \beta, \gamma), x\rangle=\langle \alpha\otimes \beta\otimes \gamma, \Delta (x)\rangle, ~~ \forall \alpha, \beta, \gamma\in A^*, ~~x\in A.
 \end{equation}

{\it A local cocycle 3-Lie bialgebra} \cite{BCM5} is a pair $(A, \Delta)$, where $A$  is a 3-Lie algebra,  and
$$\Delta=\Delta_1+\Delta_2+\Delta_3:A\rightarrow A\wedge
A\wedge A$$ is a linear mapping satisfying that 
\begin{itemize}
\item $(A^*, \Delta^*)$ is a 3-Lie algebra, 
\item $\Delta_{1}$ is a $1$-cocycle associated to the $A$-module $(A\otimes A\otimes A,ad\otimes 1\otimes 1)$,
\item $\Delta_{2}$ is a $1$-cocycle associated to the $A$-module $(A\otimes A\otimes A,1\otimes ad\otimes 1)$,
\item $\Delta_{3}$ is a $1$-cocycle associated to the $A$-module $(A\otimes A\otimes A,1\otimes 1\otimes ad)$.
\end{itemize}

For  $r=\sum_i x_i\otimes y_i\in A\otimes A$, denotes
\begin{equation}
\begin{split}
[[r,r,r]]: \equiv&\sum_{i,j,k}\big([x_i,x_j,x_k]\otimes y_i\otimes y_j\otimes y_k+x_i\otimes [y_i,x_j,x_k]\otimes y_j\otimes y_k\\
&+ x_i\otimes x_j\otimes [y_i, y_j,x_k]\otimes y_k+ x_i\otimes x_j\otimes x_k\otimes [y_i,y_j,y_k]\big).
\end{split}
\label{eq:rrr}
\end{equation}
The equation
\begin{equation}
[[r,r,r]]=0
\label{eq:CYBE}
\end{equation}
is called the {\it 3-Lie classical Yang-Baxter equation} in the 3-Lie algebra $A$, and simply denoted by {\it CYBE}.

Let $A$ be a 3-Lie algebra with a basis $\{x_{1}, \cdots, x_{n}\}$, $\{x^*_{1},$ $\cdots, $ $x^*_{n}\}$ the dual basis of $A^*$, and $D$ an involutive derivation of $A$. Define the tensor $\overline D\in A^{*}\otimes A$ by,
\begin{equation}
\overline{D}(x,\xi)=\langle \xi, Dx\rangle, ~~ \forall x\in A, \xi \in A^{*}.
\end{equation}
Thanks to Theorem 3.3 in \cite{involutive}
\begin{equation}\label{eq:barD}
r=\overline{D}-\sigma_{12}\overline{D}
 \end{equation}
is a skew-symmetric solution  of  {\it CYBE}  in the semi-direct product 3-Lie algebra $A\ltimes_{ad^*} A^*$, and
\begin{equation}\label{eq:Dr}
\overline{D}=\sum_{i=1}^n x_{i}^{*}\otimes Dx_i, ~~~~ r=\sum_{i=1}^n x_{i}^{*}\otimes Dx_i-\sum_{i=1}^n Dx_{i}\otimes x^*_i\in A^{*}\otimes A,
\end{equation}
where $\sigma_{12}$ is the exchanging mapping. And $r$ in \eqref{eq:Dr} induces a local cocycle 3-Lie bialgebra $(A\ltimes_{ad^*} A^*, \Delta)$ on the semi-direct product 3-Lie algebra $(A\ltimes_{ad^*} A^*, \mu)$,  where $\forall x\in A\ltimes_{ad^*} A^*,$
\begin{equation}\label{eq:Delta1}
\begin{array}{l}
 \left\{\begin{array}{l}
\Delta_1(x)=\sum\limits_{i,j=1}^n [x,x_i^{*},-Dx_j]\otimes x_j^{*}\otimes Dx_i+\sum\sum\limits_{i,j=1}^n[x,-Dx_i^{*},x_j^{*}]\otimes Dx_j^{*}\otimes x_i^{*}\\
\vspace{2mm}\hspace{1.5cm}+ \sum\limits_{i,j=1}^n[x,Dx_i^{*},Dx_j]\otimes x_j^{*}\otimes x_i^{*},\\
\vspace{2mm}\Delta_2(x)=\sigma_{13}\sigma_{12}\Delta_1(x),\\
\vspace{2mm}\Delta_3(x)=\sigma_{12}\sigma_{13}\Delta_1(x),\\
\vspace{2mm}\Delta(x)~=\Delta_1(x)+\Delta_2(x)+\Delta_3(x).\\
\end{array}\right.
\end{array}
\end{equation}

If we suppose that $x_1, \cdots, x_s\in A_1$ and  $x_{s+1}, \cdots, x_n\in A_{-1}$. By a direct computation we have
\begin{equation}\label{eq:999}
\begin{aligned}
\Delta_1(x)&=\sum_{i=1}^s \sum_{j=1}^s [x,x_i^{*},-x_j]\otimes x_j^{*}\otimes x_i
+\sum_{i=1}^s \sum_{j=s+1}^n [x,x_i^{*},x_j]\otimes x_j^{*}\otimes x_i\\&+\sum_{i=s+1}^n \sum_{j=1}^s [x,x_i^{*},-x_j]\otimes x_j^{*}\otimes (-x_i)+\sum_{i=s+1}^n \sum_{j=s+1}^n [x,x_i^{*},x_j]\otimes x_j^{*}\otimes (-x_i)\\&+\sum_{i=1}^s \sum_{j=1}^s [x,-x_i,x_j^{*}]\otimes x_j\otimes x_i^{*}+\sum_{i=1}^s \sum_{j=s+1}^n [x,-x_i,x_j^{*}]\otimes (-x_j)^{*}\otimes x_i^{*}\\&+\sum_{i=s+1}^n \sum_{j=1}^s [x,x_i,x_j^{*}]\otimes x_j\otimes x_i^{*}+\sum_{i=s+1}^n \sum_{j=s+1}^n [x,x_i,x_j^{*}]\otimes (-x_j)\otimes x_i^{*}\\&+\sum_{i=1}^s \sum_{j=1}^s [x,x_i,x_j]\otimes x_j^{*}\otimes x_i^{*}+\sum_{i=1}^s \sum_{j=s+1}^n [x,x_i,-x_j]\otimes x_j^{*}\otimes x_i^{*}\\&+\sum_{i=s+1}^n \sum_{j=1}^s [x,-x_i,x_j]\otimes x_j^{*}\otimes x_i^{*}+\sum_{i=s+1}^n \sum_{j=s+1}^n [x,x_i,x_j]\otimes x_j^{*}\otimes x_i^{*},\\
\end{aligned}
\end{equation}
 \begin{equation}\label{eq:9991}
\begin{aligned}
&\Delta_2(x)=\sigma_{13}\sigma_{12}\Delta_1(x)\\
=&\sum_{i=1}^s \sum_{j=1}^s x_i\otimes[x,x_i^{*},-x_j]\otimes x_j^{*}+\sum_{i=1}^s \sum_{j=s+1}^n x_i\otimes[x,x_i^{*},x_j]\otimes x_j^{*}\\
-&\sum_{i=s+1}^n \sum_{j=1}^s  x_i\otimes[x,x_i^{*},-x_j]\otimes x_j^{*}-\sum_{i=s+1}^n \sum_{j=s+1}^n x_i\otimes[x,x_i^{*},x_j]\otimes x_j^{*} \\
-&\sum_{i=1}^s \sum_{j=1}^s x_i^{*}\otimes[x,x_i,x_j^{*}]\otimes x_j +\sum_{i=1}^s \sum_{j=s+1}^n x_i^{*}\otimes[x,x_i,x_j^{*}]\otimes x_j^{*}\\
+&\sum_{i=s+1}^n \sum_{j=1}^sx_i^{*}\otimes [x,x_i,x_j^{*}]\otimes x_j-\sum_{i=s+1}^n \sum_{j=s+1}^n x_i^{*}\otimes[x,x_i,x_j^{*}]\otimes x_j
\\+&\sum_{i=1}^s \sum_{j=1}^s  x_i^{*}\otimes[x,x_i,x_j]\otimes x_j^{*}-\sum_{i=1}^s \sum_{j=s+1}^n x_i^{*} \otimes [x,x_i,x_j]\otimes x_j^{*}\\
-&\sum_{i=s+1}^n \sum_{j=1}^s x_i^{*} \otimes[x,x_i,x_j]\otimes x_j^{*}+\sum_{i=s+1}^n \sum_{j=s+1}^n x_i^{*} \otimes[x,x_i,x_j]\otimes x_j^{*},\\
 \end{aligned}
\end{equation}
 \begin{equation}\label{eq:9992}
\begin{aligned}
 &\Delta_3(x)=\sigma_{12}\sigma_{13}\Delta_1(x)\\
 =&\sum_{i=1}^s \sum_{j=1}^s x_j^{*}\otimes x_i\otimes[x,x_i^{*},-x_j]
+\sum_{i=1}^s \sum_{j=s+1}^n  x_j^{*}\otimes x_i\otimes[x,x_i^{*},x_j]\\
+&\sum_{i=s+1}^n \sum_{j=1}^s x_j^{*}\otimes x_i\otimes[x,x_i^{*},x_j] -\sum_{i=s+1}^n \sum_{j=s+1}^n x_j^{*}\otimes x_i\otimes[x,x_i^{*},x_j]\\
-&\sum_{i=1}^s \sum_{j=1}^s x_j\otimes x_i^{*}\otimes[x,x_i,x_j^{*}]+\sum_{i=1}^s \sum_{j=s+1}^n  x_j^{*}\otimes x_i^{*}\otimes[x,x_i,x_j^{*}]\\
+&\sum_{i=s+1}^n \sum_{j=1}^s  x_j\otimes x_i^{*}\otimes[x,x_i,x_j^{*}]-\sum_{i=s+1}^n \sum_{j=s+1}^n x_j\otimes x_i^{*}\otimes[x,x_i,x_j^{*}]\\
+&\sum_{i=1}^s \sum_{j=1}^s  x_j^{*}\otimes x_i^{*}\otimes[x,x_i,x_j]+\sum_{i=1}^s \sum_{j=s+1}^n x_j^{*}\otimes x_i^{*}\otimes[x,x_i,-x_j] \\
-&\sum_{i=s+1}^n \sum_{j=1}^s  x_j^{*}\otimes x_i^{*}\otimes[x,x_i,x_j]+\sum_{i=s+1}^n \sum_{j=s+1}^n  x_j^{*}\otimes x_i^{*}\otimes[x,x_i,x_j].\\
 \end{aligned}
\end{equation}

 For convenience, in the following, the semi-direct product 3-Lie algebra $(A\ltimes_{ad^*}A^*, \mu)$ is denoted by $B_1$, and  the 3-Lie algebra $((A\oplus A^*)^{*}, \Delta^*)$ is denoted by $B_2$, where
$$\Delta^*:(A\oplus A^*)^{*}\wedge(A\oplus A^*)^{*}\wedge (A\oplus A^*)^{*}\rightarrow (A\oplus A^*)^*$$ is the dual mapping of $\Delta=\Delta_1+\Delta_2+\Delta_3$ defined as \eqref{eq:dualdelta}.

Suppose that
\begin{equation}\label{eq:Pi1}
\Pi_{1}=\{x_{1}, \cdots, x_{s},x_{s+1},\cdots, x_{n},x_{1}^{*},\cdots, x_{s}^{*},x_{s+1}^{*},\cdots, x_{n}^{*}\}
 \end{equation}
is a basis of $B_{1}$, and

\begin{equation}\label{eq:Pi2}
\Pi_{2}=\{y_{1},\cdots, y_{s},y_{s+1},\cdots, y_{n},y_{1}^{*},\cdots, y_{s}^{*},y_{s+1}^{*},\cdots, y_{n}^{*}\}
 \end{equation}
 is a basis of $B_2$ such that $x_i\in A, x_i^*\in A^*$ for $1\leq i\leq n$ and
$x_k\in A_1$, $x_l\in A_{-1}$, $1\leq k\leq s$, $s+1\leq l\leq n$, satisfying
 \begin{equation}\label{eq:bbb}
    \langle x_{i},y_{j}^{*} \rangle= \langle x_{i}^{*},y_{j} \rangle=\langle x_i, x_j^*\rangle=\delta_{ij},
     \langle x_{i},y_{j} \rangle=\langle x_{i}^{*},y_{j}^{*} \rangle=0,~~~~~1\leq i,j\leq n,
    \end{equation}
and the multiplication of the 3-Lie algebra $A$ in the basis $x_1, \cdots x_s, x_{s+1}, \cdots, x_n$ is  \eqref{eq:ooo}.

Thanks to Theorem \ref{thm:semimu} and Eqs \eqref{eq:999}, \eqref{eq:9991}, \eqref{eq:9992} and \eqref{eq:ooo},   for $\forall 1\leq t\leq n$,
\begin{equation}\label{eq:AAA}
\small\begin{aligned}
\Delta(x_{t})=&\big(-\sum_{i,j=1}^s\sum_{k=s+1}^n+\sum_{i,k=1}^s \sum_{j=s+1}^n
+\sum_{i,k=s+1}^n\sum_{j=1}^s-\sum_{i,j=s+1}^n\sum_{k=1}^s\big)~\Gamma_{tjk}^{i}x_k^{*}\otimes x_j^{*}\otimes x_i\\
&-\big(\sum_{i,j,k=1}^s+\sum_{i,j,k=s+1}^n\big)\Gamma_{tjk}^{i}x_k^{*}\otimes x_j^{*}\otimes x_i+\big(\sum_{i,j=1}^s\sum_{k=s+1}^n-\sum_{i=1}^s \sum_{j,k=s+1}^n
\big)\Gamma_{tik}^{j}x_i^{*}\otimes x_k^{*}\otimes x_j\\
&
+\big(-\sum_{i=s+1}^n\sum_{j,k=1}^s+\sum_{i,k=s+1}^n\sum_{j=1}^s+\sum_{i,j,k=1}^s+\sum_{i,j,k=s+1}^n\big)~\Gamma_{tik}^{j}x_i^{*}\otimes x_k^{*}\otimes x_j
\\
&+\sum_{i,j,k=1}^s\Gamma_{tij}^{k}x_j^{*}\otimes x_i^{*}\otimes x_k
-\big(\sum_{i,k=1}^s\sum_{j=s+1}^n+\sum_{i=1}^s\sum_{j,k=s+1}^n\big)\Gamma_{tij}^{k}x_j^{*}\otimes x_i^{*}\otimes x_k\\
&+\big(\sum_{i=s+1}^n\sum_{j,k=1}^s-\sum_{i,k=s+1}^n\sum_{j=1}^s+\sum_{i,j=s+1}^n\sum_{k=1}^s\big)
\Gamma_{tij}^{k}x_j^{*}\otimes x_i^{*}\otimes x_k,
\end{aligned}
\end{equation}
\begin{equation}\label{eq:BBB}
\small\begin{aligned}
\Delta(x_{t}^{*})&=\sum_{i=1}^s \sum_{j=1}^s \sum_{k=s+1}^n \big(-\Gamma_{ijk}^{t})(x_{k}^{*}\otimes x_j^{*}\otimes x_i^{*}+x_i^{*}\otimes x_{k}^{*}\otimes x_j^{*}+x_j^{*}\otimes x_{i}^{*}\otimes x_k^{*}\big)\\&+
\sum_{i=1}^s \sum_{j=s+1}^n \sum_{k=1}^s \Gamma_{ijk}^{t}\big(x_{k}^{*}\otimes x_j^{*}\otimes x_i^{*}+x_i^{*}\otimes x_{k}^{*}\otimes x_j^{*}+x_j^{*}\otimes x_{i}^{*}\otimes x_k^{*}\big)\\&+
\sum_{i=1}^s \sum_{j=s+1}^n \sum_{k=s+1}^n \Gamma_{ijk}^{t}\big(x_{k}^{*}\otimes x_j^{*}\otimes x_i^{*}+x_i^{*}\otimes x_{k}^{*}\otimes x_j^{*}+x_j^{*}\otimes x_{i}^{*}\otimes x_k^{*}\big)\\&+
\sum_{i=s+1}^n \sum_{j=1}^s \sum_{k=1}^s \Gamma_{ijk}^{t}\big(x_{k}^{*}\otimes x_j^{*}\otimes x_i^{*}+x_i^{*}\otimes x_{k}^{*}\otimes x_j^{*}+x_j^{*}\otimes x_{i}^{*}\otimes x_k^{*}\big)\\&+
\sum_{i=s+1}^n \sum_{j=1}^s \sum_{k=s+1}^n \Gamma_{ijk}^{t}\big(x_{k}^{*}\otimes x_j^{*}\otimes x_i^{*}+x_i^{*}\otimes x_{k}^{*}\otimes x_j^{*}+x_j^{*}\otimes x_{i}^{*}\otimes x_k^{*}\big)\\&
\sum_{i=s+1}^n \sum_{j=s+1}^n \sum_{k=1}^s \big(-\Gamma_{ijk}^{t})(x_{k}^{*}\otimes x_j^{*}\otimes x_i^{*}+x_i^{*}\otimes x_{k}^{*}\otimes x_j^{*}+x_j^{*}\otimes x_{i}^{*}\otimes x_k^{*}\big).
\end{aligned}
\end{equation}

 Then we get the following result.

\begin{theorem}\label{thm:B2} Let $A$ be a 3-Lie algebra with an involutive derivation $D$, and let  the multiplication of  $A$ in the basis $\{x_1, \cdots, x_s,$ $ x_{s+1}, \cdots, x_n\}$ be \eqref{eq:ooo}, where $A_1=\langle x_1,$ $ \cdots, x_s\rangle$ and $A_{-1}=\langle x_{s+1},$ $ \cdots, x_n\rangle$. Then the multiplication of the 3-Lie algebra $(B_{2}, \Delta^*)$ in the basis $\{y_{1},\cdots,$ $ y_{s},$ $y_{s+1}\cdots y_{n},$ $y_{1}^{*},\cdots, y_{s}^{*},$ $y_{s+1}^{*},\cdots ,y_{n}^{*}\}$ is as follows

\begin{equation}\label{eq:MB21}
\begin{cases}
\Delta^{*}(y_a,y_b,y_c)=
\begin{cases}\begin{split}
-\sum_{k=1}^s \Gamma ^{k}_{abc}y_{k}, &~~1\leq a,b\leq s<c\leq n,\\
-\sum_{k=s+1}^n \Gamma ^{k}_{abc}y_{k},& ~~1\leq a\leq s<b,c\leq n,\\
0,~& ~~~1\leq a,b,c\leq s, \\
&~~or~~ s+1\leq a,b,c\leq n.\\
\end{split}
\end{cases}\\\\
\Delta^{*}(y_a,y_b,y_c^{*})=
\begin{cases}\begin{split}
\sum_{k=s+1}^n \Gamma ^{c}_{abk}y_{k}^{*},&~~ 1\leq a,b,c\leq s,\\
\sum_{k=1}^s \Gamma ^{c}_{abk}y_{k}^{*}, & ~~s+1\leq a,b,c\leq n,\\
\sum_{k=1}^s \Gamma ^{c}_{abk}y_{k}^{*},& ~~1\leq a,c\leq s<b\leq n,\\
\sum_{k=s+1}^n \Gamma ^{c}_{abk}y_{k}^{*}, & ~~1\leq a\leq s<b,c\leq n,\\
0, & ~~1\leq a,b\leq s<c\leq n,\\
& ~~or~1\leq c\leq s<a,b\leq n,\\
\end{split}\end{cases}\\\\
\Delta^{*}(y_a,y_b^{*},y_c^{*})=\Delta^{*}(y_a^{*},y_b^{*},y_c^{*})=0, 1\leq a, b, c\leq n.
\end{cases}
\end{equation}
\end{theorem}

\begin{proof}
Suppose
\begin{equation}\label{eq:qqq1}
\Delta^*(y_a,y_b,y_c)=\sum\limits_{k=1}^n \lambda^{k}_{abc}y_{k}+\sum\limits_{k=1}^n \mu^{k}_{abc}y_{k}^{*},~~~~ \lambda^{k}_{abc},~~ \mu^{k}_{abc}\in F,~~~ 1\leq a,b,c,k\leq n.
 \end{equation}

Thanks to \eqref{eq:bbb},
 $$\langle\Delta^*(y_a,y_b,y_c),x_{t}\rangle=\langle \sum\limits_{k=1}^n \lambda^{k}_{abc}y_{k}+\sum\limits_{k=1}^n \mu^{k}_{abc}y_{k}^{*}, x_t\rangle=\mu_{abc}^{t},~~1\leq a,b,c,t\leq n,$$
 $$\langle\Delta^*(y_a,y_b,y_c),x_{t}^{*}\rangle=\sum\limits_{k=1}^n \lambda^{k}_{abc}y_{k}+\sum\limits_{k=1}^n \mu^{k}_{abc}y_{k}^{*}, x_t^*\rangle=\lambda_{abc}^{t},~~1\leq a,b,c,t\leq n.$$

By Eqs \eqref{eq:AAA} and \eqref{eq:BBB}, if $1\leq a,b,c\leq s, 1\leq t\leq n,$ then
$$\langle \Delta^*(y_a,y_b,y_c),x_{t}\rangle=\langle y_a\otimes y_b\otimes y_c,\Delta(x_{t})\rangle=0,$$
$$\langle \Delta^*(y_a,y_b,y_c),x_{t}^{*}\rangle=\langle y_a\otimes y_b\otimes y_c,\Delta(x_{t}^{*})\rangle=0.$$
Therefore,~$\lambda_{abc}^{t}=\mu_{abc}^{t}=0,$
and $\Delta^{*}(y_a,y_b,y_c)=0.$

For $\forall ~1 \leq a,b\leq s<c\leq n, 1\leq t\leq n,$
$$\langle \Delta^*(y_a,y_b,y_c),x_{t}\rangle=\langle y_a\otimes y_b\otimes y_c,\Delta(x_{t})\rangle=0,$$
\begin{equation*}
\langle \Delta^*(y_a,y_b,y_c),x_{t}^{*}\rangle=\langle y_a\otimes y_b\otimes y_c,\Delta(x_{t}^{*})\rangle=
\begin{cases}\begin{split}
-\Gamma_{abc}^{t}, &~~1\leq t\leq s,\\
0, &~~s+1\leq t\leq n.
\end{split}
\end{cases}
\end{equation*}

Therefore, ~$\mu_{abc}^{t}=0, \lambda_{abc}^{t}=-\Gamma_{abc}^{t},$
and
$$\Delta^{*}(y_a,y_b,y_c)=-\sum\limits_{k=1}^s \Gamma ^{k}_{abc}y_{k}.$$

For $\forall ~~1\leq a\leq s<b,c\leq n, 1\leq t\leq n,$
$$\langle \Delta^*(y_a,y_b,y_c),x_{t}\rangle=\langle y_a\otimes y_b\otimes y_c,\Delta(x_{t})\rangle=0,$$
\begin{equation*}
\langle \Delta^*(y_a,y_b,y_c),x_{t}^{*}\rangle=\langle y_a\otimes y_b\otimes y_c,\Delta(x_{t}^{*})\rangle=
\begin{cases}\begin{split}
0, &~~1\leq t\leq s,\\
-\Gamma_{abc}^{t}, &~~s+1\leq t\leq n.
\end{split}
\end{cases}
\end{equation*}
\\
 we get $\mu_{abc}^{t}=0,~~ \lambda_{abc}^{t}=-\Gamma_{abc}^{t},$
and $$\Delta^{*}(y_a,y_b,y_c)=-\sum\limits_{k=s+1}^n \Gamma ^{k}_{abc}y_{k}.$$

Similarly, suppose
$$\Delta^*(y_a,y_b,y_c^{*})=\sum\limits_{k=1}^n \lambda^{k}_{abc}y_{k}+\sum\limits_{k=1}^n \mu^{k}_{abc}y_{k}^{*},~~~ \lambda^{k}_{abc}, \mu^{k}_{abc}\in F, 1\leq a,b,c,k\leq n.
 $$

Thanks to \eqref{eq:bbb},
 $$\langle\Delta^*(y_a,y_b,y_c^{*}),x_{t}\rangle=\langle y_a\otimes y_b\otimes y_c^{*}), \Delta(x_t)\rangle=\mu_{abc}^{t},~~1\leq a,b,c,t\leq n,$$
 $$\langle\Delta^*(y_a,y_b,y_c^{*}),x_{t}^{*}\rangle=\langle y_a\otimes y_b\otimes y_c^{*}), \Delta(x_t^*)\rangle=\lambda_{abc}^{t},~~1\leq a,b,c,t\leq n.$$

For $\forall ~ 1\leq a,b,c\leq s, 1\leq t\leq n,$ by Eqs \eqref{eq:AAA} and \eqref{eq:BBB},
\begin{equation*}
\langle \Delta^*(y_a,y_b,y_c^{*}),x_{t}\rangle=\langle y_a\otimes y_b\otimes y_c^{*},\Delta(x_{t})\rangle=
\begin{cases}\begin{split}
0, &~~1\leq t\leq s,\\
\Gamma_{abt}^{c}, &~~s+1\leq t\leq n.
\end{split}
\end{cases}
\end{equation*}

$$\langle \Delta^*(y_a,y_b,y_c^{*}),x_{t}^{*}\rangle=\langle y_a\otimes y_b\otimes y_c^{*},\Delta(x_{t}^{*})\rangle=0$$

Therefore,$$\lambda_{abc}^{t}=0, \mu_{abc}^{t}=\Gamma_{abt}^{c}, ~~\Delta^{*}(y_a,y_b,y_c^{*})=\sum\limits_{k=s+1}^n \Gamma ^{c}_{abk}y_{k}^{*}.$$

For $\forall ~1\leq a,c\leq s<b\leq n, 1\leq t\leq n,$
\begin{equation*}
\langle \Delta^*(y_a,y_b,y_c^{*}),x_{t}\rangle=\langle y_a\otimes y_b\otimes y_c^{*},\Delta(x_{t})\rangle=
\begin{cases}\begin{split}
\Gamma_{abt}^{c}, &~~1\leq t\leq s,\\
0, &~~s+1\leq t\leq n.
\end{split}
\end{cases}
\end{equation*}

$$\langle \Delta^*(y_a,y_b,y_c^{*}),x_{t}^{*}\rangle=\langle y_a\otimes y_b\otimes y_c^{*},\Delta(x_{t}^{*})\rangle=0$$

Therefore,$$\lambda_{abc}^{t}=0, \mu_{abc}^{t}=\Gamma_{abt}^{c} ,~~\Delta^{*}(y_a,y_b,y_c^{*})=\sum_{k=1}^s \Gamma ^{c}_{abk}y_{k}^{*}.$$

For $\forall ~ 1\leq a\leq s<b,c\leq n, 1\leq t\leq n,$
\begin{equation*}
\langle \Delta^*(y_a,y_b,y_c^{*}),x_{t}\rangle=\langle y_a\otimes y_b\otimes y_c^{*},\Delta(x_{t})\rangle=
\begin{cases}\begin{split}
0, &~~1\leq t\leq s,\\
\Gamma_{abt}^{c}, &~~s+1\leq t\leq n.
\end{split}
\end{cases}
\end{equation*}

$$\langle \Delta^*(y_a,y_b,y_c^{*}),x_{t}^{*}\rangle=\langle y_a\otimes y_b\otimes y_c^{*},\Delta(x_{t}^{*})\rangle=0.$$

Therefore,$$\lambda_{abc}^{t}=0, \mu_{abc}^{t}=\Gamma_{abt}^{c},~~ \Delta^{*}(y_a,y_b,y_c^{*})=\sum_{k=s+1}^n \Gamma ^{c}_{abk}y_{k}^{*}.$$

By the  similar discussion to the above,  we get the result for other cases. We omit the computation process.
\end{proof}

\section{Manin triples and matched pairs of 3-Lie algebras  induced by involutive derivations }

 {\it A metric}  on  a 3-Lie algebra $A$ is a non-degenerate  symmetric bilinear form  $( \cdot, \cdot): A\otimes A\rightarrow \mathbb F$ satisfying
\begin{equation}
 ([x_1,x_2, x_3], x_4)+ ([x_1,x_2, x_4], x_3)=0,~~\forall x_1,x_2, x_3,x_4\in A.
 \label{eq:jacobi}
\end{equation}
The pair $(A, ( \cdot, \cdot)) $ is called {\it a metric 3-Lie algebra}, or simply $A$ is a metric 3-Lie algebra.

If there are two subalgebras $A_1$ and $A_2$ of  $(A, ( \cdot , \cdot))$ such that $A=A_1\oplus A_2$ (as vector spaces),
$(A_1, A_1 )=0$, $(A_2, A_2)=0$, $[A_1, A_1, A_2]\subseteq A_2$ and   $[A_2, A_2, A_1]\subseteq A_1$,
then the 5-tuple
\[
    (A, [\cdot, \cdot, \cdot], ( \cdot, \cdot ), A_1, A_2) \quad (\text{ or 4-tuple } (A,  ( \cdot, \cdot ), A_1, A_2))
\]
is called {\it a Manin triple} \cite{BCM5}.

Let  $(A, ( \cdot, \cdot)_{A},A_{1},A_{2})$ be a Manin triple, and $(A', ( \cdot, \cdot)_{A'})$ be a metric 3-Lie algebra.
If there is a 3-Lie  algebra isomorphism $f: A\rightarrow A'$
satisfying $$~(x, y)_{A}=(f(x),f(y))_{A^{'}},~~~\forall x, y\in A,$$ then
$(A', ( \cdot, \cdot)_{A'}, f(A_1), f(A_2))$ is also a  Manin triple. And in this case we say that $(A,(\cdot,\cdot)_{A}, A_{1}, A_{2})$ is isomorphic to $(A',(\cdot,\cdot)_{A'}, A_{1}', A_{2}')$, where
$f(A_{1})=A_{1}^{'}, ~f(A_{2})=A_{2}^{'}. $

Let $(A,[\cdot,\cdot,\cdot])$ and $(A^{'},[\cdot,\cdot,\cdot]')$ be two 3-Lie algebras. Suppose

$$\rho: A\wedge A\rightarrow Der(A'), ~~ \mbox{ and} ~~\chi: A'\wedge A'\rightarrow Der(A)$$
are linear mappings.
For $\forall x_{i}\in A$ and $a_{i}\in A^{'}, 1\leq i\leq 3$ we give the following identities:
\begin{equation}\label{eq:22'}
\aligned
&ad_{x_1x_2}\chi(a_1, a_2)x_3-\chi(a_1, a_2)ad_{x_1x_2}x_3\\
=&\chi(\rho(x_1, x_2)a_1, a_2)x_3+\chi(a_1, \rho(x_1, x_2)a_2)x_3,
\endaligned
\end{equation}
\begin{equation}\label{eq:22''}
\aligned
&ad'_{a_1a_2}\rho(x_1, x_2)a_3-\rho(x_1, x_2)ad'_{a_1a_2}a_3\\
=&\rho(\chi(a_1, a_2)x_1, x_2)a_3+\rho(x_1, \chi(a_1, a_2)x_2)a_3,
\endaligned
\end{equation}

\begin{equation}\label{eq:33'}
\aligned
&ad_{x_1 x_2}\chi(a_1, a_2)x_3=\chi(\rho(x_1, x_2)a_1, a_2)x_3\\
-&\chi(a_1, \rho(x_3, x_1)a_2)x_2-\chi(a_1, \rho(x_2, x_3)a_2)x_1,
\endaligned
\end{equation}

\begin{equation}\label{eq:33''}
\aligned
&ad'_{a_1 a_2}\rho(x_1, x_2)a_3=\rho(\chi(a_1, a_2)x_1, x_2)a_3\\
-&\rho(x_1, \chi(a_3, a_1)x_2)a_2-\rho(x_1, \chi(a_2, a_3)x_2)a_1,
\endaligned
\end{equation}
where $(A', ad')$ is the adjoint representation of the 3-Lie algebra $(A', [\cdot, \cdot, \cdot]')$.

\begin{theorem}\label{thm:AoplusA'}
Suppose that  $(A,[\cdot,\cdot,\cdot])$ and $(A^{'},[\cdot,\cdot,\cdot]')$ are 3-Lie algebras, and there are  linear mappings
$\rho:\wedge^{2}A\rightarrow Der(A')$ and $\chi:\wedge^{2}A^{'}\rightarrow Der(A)$ such that $(A', \rho)$ is an  $A$-module, $(A, \chi)$ is an $A'$-module,  and Eqs \eqref{eq:22'}, \eqref{eq:22''}, \eqref{eq:33'} and \eqref{eq:33''} hold.
Then $(A\oplus A', [\cdot, \cdot, \cdot]_{A\oplus A'})$  is a 3-Lie algebra, where for~ $\forall x_{i}\in A, a_{i}\in A^{'}, 1\leq i\leq 3,$
\begin{equation}
\aligned \
&[x_{1}+a_{1},x_{2}+a_{2},x_{3}+a_{3}]_{A\oplus A^{'}}\\
=&[x_{1},x_{2},x_{3}]+\rho(x_{1},x_{2})a_{3}+\rho(x_{3},x_{1})a_{2}+
\rho(x_{2},x_{3})a_{1}\\
+&[a_{1},a_{2},a_{3}]^{'}+\chi(a_{1},a_{2})x_{3}+\chi(a_{3},a_{1})x_{2}+\chi(a_{2},a_{3})x_{1}.\\
\endaligned
\end{equation}
\end{theorem}

\begin{proof} Apply Proposition 4.4 in \cite{BCM5}.
\end{proof}

The  4-tuple $(A, A', \rho, \chi)$ is called  a {\it matched pair of 3-Lie algebras.}

\vspace{2mm}Let $(A,[\cdot,\cdot,\cdot])$ and $(A^{*},[\cdot,\cdot,\cdot]_{*})$ be 3-Lie algebras, where $A^*$ is the dual space of $A$. There is a natural
non-degenerate symmetric bilinear form $( \cdot , \cdot ): (A\oplus A^{*})\otimes (A\oplus A^{*})\rightarrow \mathbb F$ given by
\begin{equation}\label{eq:+}
(x_1+\xi,x_2+\eta)=\langle x_1,\eta\rangle+\langle\xi,x_2\rangle, \forall x_1, x_2\in A, \xi,\eta \in A^{*}.
\end{equation}
Then $(A, A)=(A^*, A^*)=0.$

Define linear multiplication   $[\cdot,\cdot,\cdot]_{A\oplus A^{*}} :(A\oplus A^{*})^{\wedge 3}\rightarrow A\oplus A^{*}$ by, $\forall x_i\in A,$ $y_i\in A^*$ for $1\leq i\leq 3,$
\begin{equation}\label{eq:oplus}
\aligned \
&[x_{1}+y_{1},x_{2}+y_{2},x_{3}+y_{3}]_{A\oplus A^{*}}\\
=&[x_{1},x_{2},x_{3}]+ad^{*}_{x_{1}x_{2}}y_{3}+ad^{*}_{x_{3}x_{1}}y_{2}
+ad^{*}_{x_{2}x_{3}}y_{1}\\
+&a\phi^{*}_{y_{1}y_{2}}x_{3}+a\phi^{*}_{y_{3}y_{1}}x_{2}
+a\phi^{*}_{y_{2}y_{3}}x_{1}+[y_{1},y_{2},y_{3}]_{*},\\
\endaligned
\end{equation}
where $(A^*, ad^*)$ is the coadjoint representation of the 3-Lie algebra $(A, [\cdot, \cdot, \cdot]$, and  $(A, a\phi^*)$ is the coadjoint representation of the 3-Lie algebra $(A^*,  [\cdot, \cdot, \cdot]_*)$.

Then by \eqref{eq:+} and \eqref{eq:oplus}, for $\forall x_i\in A,$ $y_i\in A^*$ for $1\leq i\leq 4,$

\begin{equation*}
\aligned
&([x_{1}+y_{1},x_{2}+y_{2},x_{3}+y_{3}]_{A\oplus A^{*}},x_4+y_4)+ (x_{3}+y_{3}, [x_{1}+y_{1},x_{2}+y_{2}, x_4+y_4]_{A\oplus A^{*}})\\
=&([x_{1},x_{2},x_3]+a\phi^{*}_{y_1y_2}x_3+a\phi^{*}_{y_3y_1}x_{2}
+a\phi^{*}_{y_2y_3}x_1, y_4)\\
&+(ad^{*}_{x_1,x_2}y_3+ad^{*}_{x_3x_1}y_2
+ad^{*}_{x_2x_3}y_1+[y_1,y_2,y_3]_{*}, x_4)\\
&+(y_3, [x_1,x_2,x_4]+a\phi^{*}_{y_1y_2}x_4+a\phi^{*}_{y_4y_1}x_2
+a\phi^{*}_{y_2y_4}x_1)\\
&+(x_3, ad^{*}_{x_1x_2}y_4+ad^{*}_{x_4x_1}y_2
+ad^{*}_{x_2x_4}y_1+[y_1,y_2,y_4]_{*})\\
=&([x_1,x_2,x_3], y_4)-(x_3, [y_1,y_2, y_4]_*)-(x_2, [y_3,y_1, y_4]_*)-(x_1, [y_2,y_3, y_4]_*)\\
&+(x_4,[y_1,y_2,y_3]_*)-([x_1,x_2, x_4], y_3)-([x_3,x_1, x_4], y_2)-([x_2,x_3, x_4], y_1)\\
&+([x_1,x_2,x_3], y_4)-(x_4, [y_1,y_2, y_3]_*)-(x_2, [y_4,y_1, y_3]_*)-(x_1,[y_2,y_4, y_3]_*)\\
&+(x_3, [y_1,y_2,y_4]_*)-([x_1,x_2, x_3], y_4)-([x_4,x_1, x_3], y_2)-([x_2,x_4, x_3], y_1)=0,\\
\endaligned
\end{equation*}
 and $(A, A)=(A^*, A^*)=0.$

We find that the non-degenerate symmetric bilinear form $( \cdot , \cdot )$ defined by \eqref{eq:+} satisfies product invariance, and  $A$, $A^*$ are isotropic in the 3-algebra $(A\oplus A^{*},[\cdot,\cdot,\cdot]_{A\oplus A^{*}})$. Therefore, if $(A\oplus A^{*},[\cdot,\cdot,\cdot]_{A\oplus A^{*}})$ is a $3$-Lie algebra, then  $(A\oplus A^{*},[\cdot,\cdot,\cdot]_{A\oplus A^{*}}, ( \cdot , \cdot ))$ is a metric 3-Lie algebra with isotropic
subalgebras $A$ and $A^{*}$,  and $(A\oplus A^{*}, [\cdot, \cdot, \cdot]_{A\oplus A^*}, (\cdot,\cdot), A, A^{*})$ is a Manin triple, which is called {\it the standard Manin triple.}

\vspace{4mm}
Thanks to Theorem \ref{thm:AoplusA'} and Eq \eqref{eq:oplus}, for arbitrary two 3-Lie algebras
$(A,[\cdot,\cdot,\cdot])$ and   $(A^{*},[\cdot,$ $\cdot,\cdot]_{*})$, the 5-tuple  $(A\oplus A^{*},$ $ [\cdot, \cdot, \cdot]_{A\oplus A^*}, $ $(\cdot,\cdot), A, A^{*})$
is a standard Manin triple if and only if $(A, A^{*}, ad^{*}, a\phi^{*})$ is a matched pair.

\vspace{2mm}  In the following we suppose that $A$ is a 3-Lie algebra with  an involutive derivation $D$, and  the multiplication of $A$ in the basis $\{ x_1, $ $\cdots, x_s, $ $x_{s+1},$ $ \cdots, x_n \}$ is \eqref{eq:ooo}, where  $x_i\in A_1, x_j\in A_{-1}$, $1\leq i\leq s,$ $s+1\leq j\leq n$;
 $B_1=(A\ltimes A^*, \mu)$ is the semi-direct product 3-Lie algebra of $A$ with the multiplication \eqref{eq:1} and \eqref{eq:2} in the basis $\Pi_1$ (defined by \eqref{eq:Pi1} ),  and $B_{2}=((A\oplus A^*)^{*}, \Delta^*)$ is the  3-Lie algebra induced by the local cocycle 3-Lie bialgebra $(A\oplus A^*, \Delta)$  with the multiplication \eqref{eq:MB21} in the basis $\Pi_2$
 (defined by \eqref{eq:Pi2} ).

\vspace{2mm}Let
$$a\delta^{*}: B_{1}\wedge B_{1}\rightarrow gl(B_{2}), $$
be the coadjoint representations of the 3-Lie algebra $(B_{1}, \mu)$ on $B_{2}=(A\oplus A^*)^*=B_1^*$, and
$$ a\psi^{*}:B_{2}\wedge B_{2}\rightarrow gl(B_{1}=B_2^*)$$
  be  the coadjoint representation of the 3-Lie algebra $(B_{2}, \Delta^*)$ on $B_1=B_2^*$. 
  
  Then  for $\forall y_{a}, y_{b}, y_{c},$ $y^{*}_{a}, y^{*}_{c},$ $ y_{t}, y_{r}^{*}\in \Pi_2$, and $\forall x_{a}, x_{b}, x_{c},$ $x_{e}^{*}, x^{*}_{d},$ $ x_{t},x_{t}^{*}\in \Pi_{1}$, we have
    \begin{equation}\label{eq:qqq1}
 \begin{array}{llll}
 \left\{\begin{array}{l}
 \vspace{2mm}\langle a\delta^{*}_{x_{a}x_{b}}y_{t},x_{c}\rangle =-\langle y_{t},\mu (x_{a},x_{b},x_{c}) \rangle,\\
 \vspace{2mm}\langle a\delta^{*}_{x_{a}x_{b}}y_{t},x^{*}_{d}\rangle =-\langle y_{t},\mu (x_{a},x_{b},x_{d}^{*}) \rangle,\\
  \vspace{2mm}\langle a\delta^{*}_{x_{a}x_{b}}y_{r}^{*},x_{c}\rangle =-\langle y_{r}^{*},\mu (x_{a},x_{b},x_{c}) \rangle,\\
  \vspace{2mm}\langle a\delta^{*}_{x_{a}x_{b}}y_{r}^{*},x^{*}_{d}\rangle =-\langle y_{r}^{*},\mu (x_{a},x_{b},x_{d}^{*}) \rangle,\\
  \vspace{2mm}\langle a\delta^{*}_{x_{d}^{*}x_{b}}y_{t},x_{c}\rangle =-\langle y_{t},\mu (x_{d}^{*},x_{b},x_{c}) \rangle,\\
   \vspace{2mm}\langle a\delta^{*}_{x_{d}^{*}x_{b}}y_{t},x^{*}_{e}\rangle =-\langle y_{t},\mu (x_{d}^{*},x_{b},x_{e}^{*}) \rangle.
 \end{array}\right.
\end{array}
 \end{equation}
\begin{equation}\label{eq:ppp1}
 \begin{array}{llll}
 \left\{\begin{array}{l}
\vspace{2mm} \langle a\psi^{*}_{y_{a}y_{b}}x_{t},y_{c}\rangle =-\langle x_{t},\Delta^{*} (y_{a},y_{b},y_{c}) \rangle,\\
\vspace{2mm} \langle a\psi^{*}_{y_{a}y_b}x_{t},y^{*}_{c}\rangle =-\langle x_{t},\Delta^{*} (y_{a},y_{b},y_{c}^{*}) \rangle,\\
 \vspace{2mm}\langle a\psi^{*}_{y_{a}y_{b}}x_{t}^{*},y_{c}\rangle =-\langle x_{t}^{*},\Delta^{*} (y_{a},y_{b},y_{c}) \rangle,\\
 \vspace{2mm} \langle a\psi^{*}_{y_{a}y_{b}}x_{t}^{*},y^{*}_{c}\rangle =-\langle x_{t}^{*},\Delta^{*} (y_{a},y_{b},y_{c}^{*}) \rangle,\\
 \vspace{2mm} \langle a\psi^{*}_{y_{a}^{*}y_{b}}x_{t},y_{c}\rangle =-\langle x_{t},\Delta^{*} (y_{a}^{*},y_{b},y_{c}) \rangle,\\
\vspace{2mm}\langle a\psi^{*}_{y_{a}^{*}y_{b}}x_{t},y^{*}_{c}\rangle =-\langle x_{t},\Delta^{*} (y_{a}^{*},y_{b},y_{c}^{*}) \rangle.\\
 \end{array}\right.
\end{array}
 \end{equation}

 By the above notations, we have the following result.
\begin{lemma}\label{lem:B1B2} The multiplication $[\cdot, \cdot, \cdot]_1: \wedge^3(B_1\oplus B_2)\rightarrow B_1\oplus B_2 $  of the semi-direct product 3-Lie algebra $B_1\ltimes_{a\delta^{*}}B_2=(B_1\oplus B_2, [ \cdot, \cdot, \cdot]_1)$  in the basis
$\Pi_1\cup\Pi_2=\{x_1, \cdots, x_s, x_{s+1},$ $ \cdots, x_n$, $x_1^*, \cdots, x_s^*,$ $ x_{s+1}^*, \cdots, x_n^*$, $y_1, \cdots,$ $ y_s, y_{s+1},$ $ \cdots, y_n$, $y_1^*, $ $\cdots, y_s^*,$ $ y_{s+1}^*, \cdots, y_n^*\}$, is as following, for $\forall 1\leq a, b,$ $ c, d, e, g, $ $p, q, r, f, h, t\leq n,$

$$[x_{a},x_{b},x_{c}]_1=\mu(x_a, x_b, x_c),~~ [x_{a},x_{b},x_{d}^*]_1=\mu(x_a, x_b, x_d^*),
$$are defined as \eqref{eq:1} and \eqref{eq:2},
and
$$[x_{a},x_{d}^*,x_{e}^*]_1=[x_d^*, x_e^*, x_g^*]_1=[x_{a},y_{p},y_{q}]_1=[x_{a},y_{p},y_{t}^*]_1=0,$$
$$[x_{a},y_{f}^*, y_{t}^*]_1=[y_{p},y_{q},y_{r}]_1=[y_{p},y_{q}^*, y_{t}^*]_1=[y_f^*, y_h^*, y_t^*]_1=0,$$

\begin{equation}\label{eq:coadjoint1}
\begin{cases}\begin{split}
{[}x_a,x_b,y_t]_1=a\delta^*_{x_ax_b}y_t=&
\begin{cases}\begin{split}
\sum_{k=1}^s \Gamma ^{k}_{abt}y_{k}, &~~1\leq a,b\leq s<t\leq n,\\
\sum_{k=1}^s \Gamma ^{k}_{abt}y_{k}, &~~1\leq a,t\leq s<b\leq n,\\
\sum_{k=s+1}^n \Gamma ^{k}_{abt}y_{k}, &~~1\leq t\leq s<a,b\leq n,\\
\sum_{k=s+1}^n \Gamma ^{k}_{abt}y_{k}, &~~1\leq a\leq s<b,t\leq n,\\
0,~ &~~1\leq a,b,t\leq s ~~\\
&~~or~~ s+1\leq a,b,t\leq n;\\
\end{split}\end{cases}\\
{[}x_a,x_b,y_t^{*}]_1=a\delta^*_{x_ax_b}y^*_t=&
\begin{cases}
\begin{split}
-\sum_{k=s+1}^n \Gamma ^{t}_{abk}y_{k}^{*}, &~~1\leq a,b,t\leq s,\\
-\sum_{k=1}^s \Gamma ^{t}_{abk}y_{k}^{*},&~~s+1\leq a,b,t\leq n,\\
-\sum_{k=1}^s \Gamma ^{t}_{abk}y_{k}^{*},&~1\leq a,t\leq s<b\leq n, ~\\
-\sum_{k=s+1}^n \Gamma ^{t}_{abk}y_{k}^{*},&~ 1\leq a\leq s<b,t\leq n,\\
0,&~1\leq a,b\leq s<t\leq n,~\\
&~~or~1\leq t\leq s<a,b\leq n;\\
\end{split}
\end{cases}\\
{[}x_a^{*},x_b,y_t]_1=a\delta^*_{x^*_ax_b}y_t=&
\begin{cases}
\begin{split}
\sum_{k=s+1}^n \Gamma ^{a}_{bkt}y_{k}^{*},&~~ 1\leq a,b,t\leq s,\\
\sum_{k=1}^s \Gamma ^{a}_{bkt}y_{k}^{*},&~~s+1\leq a,b,t\leq n,\\
\sum_{k=1}^s \Gamma ^{a}_{bkt}y_{k}^{*},&~~1\leq a,b\leq s<t\leq n,~\\
&~~or~1\leq a,t\leq s<b\leq n,\\
\sum_{k=s+1}^n \Gamma ^{a}_{bkt}y_{k}^{*},&~~1\leq t\leq s<a,b\leq~n,\\
&~~or~1\leq b\leq s<a,t\leq n,\\\\
0,&~~~~1\leq a\leq s<b,t\leq n,~\\
&~~~~or~1\leq b,t\leq s<a\leq n.
\end{split}\end{cases}
\end{split}\end{cases}
\end{equation}

\end{lemma}

\begin{proof} For $\forall 1\leq a, b, c, t\leq n$, suppose
$$[x_a,x_b,y_t]_1=\sum\limits_{k=1}^n \lambda^{k}_{abt}y_{k}+\sum_{k=1}^n \nu^{k}_{abt}y_{k}^{*}, ~~ \lambda^{k}_{abt}, \nu^{k}_{abt}\in F.$$

By Eq \eqref{eq:bbb}, if $1\leq a,b,c\leq n, 1\leq t\leq n,$ then

$$\langle a\delta^{*}_{x_{a}x_{b}}y_{t},x_{c}\rangle=\langle \sum\limits_{k=1}^n \lambda^{k}_{abt}y_{k}+\sum\limits_{k=1}^n \nu^{k}_{abt}y_{k}^{*},x_{c}\rangle=\nu_{abt}^{c},$$
$$\langle a\delta^{*}_{x_{a}x_{b}}y_{t},x^{*}_{c}\rangle=\langle\sum\limits_{k=1}^n \lambda^{k}_{abt}y_{k}+\sum\limits_{k=1}^n \nu^{k}_{abt}y_{k}^{*},x_{c}^{*}\rangle=\lambda_{abt}^{c}.$$

Thanks to Eqs \eqref{eq:1}, \eqref{eq:2}, \eqref{eq:bbb} and \eqref{eq:qqq1}, for all $1\leq c \leq n,$ 
$$\langle a\delta^{*}_{x_{a}x_{b}}y_{t},x_{c}\rangle =-\langle y_{t},\mu (x_{a},x_{b},x_{c}) \rangle=0, ~~ 1\leq a,b\leq s,1\leq t \leq n,$$
\begin{equation*}
\langle a\delta^*_{x_{a}x_{b}}y_{t},x^{*}_{c}\rangle =-\langle y_{t},\mu (x_{a},x_{b},x_{c}^{*}) \rangle=
\begin{cases}\begin{split}
0, &~~1\leq a,b\leq s,~~1\leq t\leq s,\\
\Gamma_{abt}^{c}, &~~1\leq a,b\leq s,, ~~s+1\leq t\leq n.
\end{split}
\end{cases}
\end{equation*}
Therefore,  $\nu_{abt}^{c}=0,\lambda_{abt}^{c}=\Gamma_{abt}^{c}$ for $1\leq a,b\leq s, s+1\leq t\leq n,$ and
\begin{equation*}
[x_a,x_b,y_t]_{1}=
\begin{cases}\begin{split}
0, &~~1\leq a,b,t\leq s,\\
\sum\limits_{k=1}^s \Gamma ^{k}_{abt}y_{k}, &~~1\leq a,b\leq s<t\leq n.
\end{split}
\end{cases}
\end{equation*}

In the case $ s+1\leq a,b\leq n,$ $1\leq t\leq n,$  we have that for all $1\leq c\leq n,$
$$\langle a\delta^{*}_{x_{a}x_{b}}y_{t},x_{c}\rangle =-\langle y_{t},\mu (x_{a},x_{b},x_{c}) \rangle=0,$$
\begin{equation*}
\langle a\delta^*_{x_{a}x_{b}}y_{t},x^{*}_{c}\rangle =-\langle y_{t},\mu (x_{a},x_{b},x_{c}^{*}) \rangle=
\begin{cases}\begin{split}
\Gamma_{abt}^{c}, &~~1\leq t\leq s,\\
0, &~~s+1\leq t\leq n,
\end{split}
\end{cases}
\end{equation*}
and $\nu_{abt}^{c}=0,\lambda_{abt}^{c}=\Gamma_{abt}^{c},$ $1\leq t\leq s.$ Therefore,

\begin{equation*}
[x_a,x_b,y_t]_{1}=
\begin{cases}\begin{split}
0, &~~s+1\leq a,b,t\leq n,\\
\sum_{k=s+1}^n \Gamma_{abt}^{k}y_{k}, &~~1\leq t\leq s<a,b\leq n.
\end{split}
\end{cases}
\end{equation*}

If  $\forall 1\leq a\leq s<b\leq n,$ $1\leq t\leq n,$ then for all $1\leq c\leq n,$ 
$$\langle a\delta^{*}_{x_{a}x_{b}}y_{t},x_{c}\rangle =-\langle y_{t},\mu (x_{a},x_{b},x_{c}) \rangle=0,$$
\begin{equation*}
\langle a\delta^*_{x_{a}x_{b}}y_{t},x^{*}_{c}\rangle =-\langle y_{t},\mu (x_{a},x_{b},x_{c}^{*}) \rangle=
\begin{cases}\begin{split}
\Gamma_{abt}^{c}, &~~1\leq t\leq s,\\
\Gamma_{abt}^{c}, &~~s+1\leq t\leq n,
\end{split}
\end{cases}
\end{equation*}
and $\nu_{abt}^{c}=0,\lambda_{abt}^{c}=\Gamma_{abt}^{c}, $ $1\leq t\leq n.$ Therefore,

\begin{equation*}
[x_a,x_b,y_t]_{1}=
\begin{cases}\begin{split}
\sum_{k=1}^s \Gamma_{abt}^{k}y_{k}, &~~s+1\leq a,t\leq s<b\leq n,\\
\sum_{k=s+1}^n \Gamma_{abt}^{k}y_{k}, &~~1\leq a\leq s<b,t\leq n.
\end{split}
\end{cases}
\end{equation*}

  By  a completely similar discussion to the above, for others cases, we get the expression of $[x_a,x_b,y_t^{*}]_1$ and $[x_a^*,x_b,y_t^{*}]_1$, respectively.
 We omit the calculation process. The identities in
\eqref{eq:coadjoint1} follow.
\end{proof}

\begin{lemma}\label{lem:B2B1} The multiplication $[\cdot, \cdot, \cdot]_2: \wedge^3(B_2\oplus B_1)\rightarrow B_2\oplus B_1 $ of  the semi-direct product 3-Lie algebra $B_2\ltimes_{a\psi^{*}}B_1=(B_1\oplus B_2, [ \cdot, \cdot, \cdot]_2)$  in the basis

\vspace{5mm}\hspace{5mm}$\Pi_1\cup\Pi_2=\{x_1, \cdots, x_s, x_{s+1},$ $ \cdots, x_n$, $x_1^*, \cdots, x_s^*,$ $ x_{s+1}^*, \cdots, x_n^*$,

\vspace{5mm}\hspace{23mm}$y_1, \cdots,$ $ y_s, y_{s+1},$ $ \cdots, y_n$, $y_1^*, $ $\cdots, y_s^*,$ $ y_{s+1}^*, \cdots, y_n^*\}$,

\vspace{5mm}\hspace{-5mm} is as follows, ~for $\forall ~~ 1\leq a, b, c, $ $d, e, g,$ $ p, q, r, $ $f, h, t\leq n,$

$$
[y_{a},y_{b},y_{c}]_2=\Delta^*(y_a, y_b, y_c),~~~ [y_{a},y_{b},y_{d}^*]_2=\Delta^*(y_a, y_b, y_d^*)
$$   are defined as \eqref{eq:MB21}, ~

$$[y_{a},y_{d}^*,y_{e}^*]_2=[y_d^*, y_e^*, y_g^*]_2=[y_{a},x_{p},x_{q}]_2=[y_{a},x_{p},x_{t}^*]_2=0,$$

$$[y_{a},x_{f}^*, x_{t}^*]_2=[x_{p},x_{q},x_{r}]_2=[x_{p},x_{q}^*, x_{t}^*]_2=[x_f^*,x_h^*, x_t^*]_2=0,$$
and
\begin{equation}\label{eq:coadjoint411}
{[}y_a,y_b,x_t]_2=a\psi^*_{y_ay_b}x_t=
\begin{cases}\begin{split}
-\sum_{k=1}^s \Gamma ^{k}_{abt}x_{k},&~ 1\leq a,b\leq s<t\leq n,\\
-\sum_{k=1}^s \Gamma ^{k}_{abt}x_{k},&~ 1\leq a,t\leq s<b\leq n,\\
-\sum_{k=s+1}^n \Gamma ^{k}_{abt}x_{k},&~ 1\leq t\leq s<a,b\leq n,\\
-\sum_{k=s+1}^n \Gamma ^{k}_{abt}x_{k},&~ 1\leq a\leq s<b,t\leq n,\\
0,&~~ 1\leq a,b,t\leq s,\\
&~~ ~or~ s+1\leq a,b,t\leq n;\\
\end{split}
\end{cases}
\end{equation}
\begin{equation}\label{eq:coadjoint412}
{[}y_a,y_b,x_t^{*}]_2=a\psi^*_{y_ay_b}x^*_t=
\begin{cases}\begin{split}
\sum_{k=s+1}^n \Gamma ^{t}_{abk}x_{k}^{*},&~ 1\leq a,b,t\leq s,\\
\sum_{k=1}^s \Gamma ^{t}_{abk}x_{k}^{*},&~s+1\leq a,b,t\leq n,\\
\sum_{k=1}^s \Gamma ^{t}_{abk}x_{k}^{*},&~1\leq a,t\leq s<b\leq n,\\
\sum_{k=s+1}^n \Gamma ^{t}_{abk}x_{k}^{*},&~ 1\leq a\leq s<b,t\leq n,\\
0,&~1\leq a,b\leq s<t\leq n,\\
&~~ ~or~ 1\leq t\leq s<a,b\leq n;
\end{split}\end{cases}
\end{equation}
\begin{equation}\label{eq:coadjoint413}
{[}y_a^{*},y_b,x_t]_2=a\psi^*_{y^*_ay_b}x_t=
\begin{cases}\begin{split}
-\sum_{k=s+1}^n \Gamma ^{a}_{bkt}x_{k}^{*},&~ 1\leq a,b,t\leq s,\\
-\sum_{k=1}^s \Gamma ^{a}_{bkt}x_{k}^{*},&~ s+1\leq a,b,t\leq n,\\
-\sum_{k=1}^s \Gamma ^{a}_{bkt}x_{k}^{*},&~ 1\leq a,b\leq s<t\leq n,\\
&~~or~1\leq a,t\leq s<b\leq n,\\
-\sum_{k=s+1}^n \Gamma ^{a}_{bkt}x_{k}^{*},&~ 1\leq t\leq s<a,b\leq n,\\&~~
or~1\leq b\leq s<a,t\leq n,\\\\
0,&~1\leq a\leq s<b,t\leq n,\\&~~~or~1\leq b,t\leq s<a\leq n.
\end{split}
\end{cases}
\end{equation}
\end{lemma}

\begin{proof} For $\forall  1\leq a, b, t\leq n,$ suppose
$$[y_a,y_b,x_t]_2=\sum_{k=1}^n \lambda^{k}_{abt}x_{k}+\sum_{k=1}^n \nu^{k}_{abt}x_{k}^{*}, ~~~\lambda^{k}_{abt}, \nu^{k}_{abt}\in F.$$

 By Eqs \eqref{eq:AAA} and  \eqref{eq:BBB}, for $1\leq a,b,c\leq n, 1\leq t\leq n,$ we have

$$\langle a\psi^{*}_{y_{a}y_{b}}x_{t},y_{c}\rangle=\langle\sum_{k=1}^n \lambda^{k}_{abt}x_{k}+\sum_{k=1}^n \nu^{k}_{abt}x_{k}^{*},y_{c}\rangle=\nu_{abt}^{c},$$
$$\langle a\psi^{*}_{y_{a}y_{b}}x_{t},y^{*}_{c}\rangle=\langle\sum_{k=1}^n \lambda^{k}_{abt}x_{k}+\sum_{k=1}^n \nu^{k}_{abt}x_{k}^{*},y_{c}^{*}\rangle=\lambda_{abt}^{c}.$$

 Thanks to Eqs \eqref{eq:bbb}, \eqref{eq:MB21} and \eqref{eq:ppp1},
for $\forall y_{a},y_{b},y_{c}\in \Pi_2, $ $x_{t}\in \Pi_{1}$, if $1\leq a,b\leq s$, $1\leq c,t\leq n$, then
$$\langle a\psi^{*}_{y_{a}y_{b}}x_{t},y_{c}\rangle =-\langle x_{t},\Delta^{*} (y_{a},y_{b},y_{c}) \rangle=0,$$
\begin{equation*}
\langle a\psi^{*}_{y_{a}y_{b}}x_{t},y^{*}_{c}\rangle =-\langle x_{t},\Delta^{*} (y_{a},y_{b},y_{c}^{*}) \rangle=
\begin{cases}\begin{split}
0, &~~1\leq t\leq s,\\
-\Gamma_{abt}^{c}, &~~s+1\leq t\leq n.
\end{split}
\end{cases}
\end{equation*}

Therefore,  $\nu_{abt}^{c}=0,\lambda_{abt}^{c}=-\Gamma_{abt}^{c},$ and
\begin{equation*}
[y_a,y_b,x_t]_2=
\begin{cases}\begin{split}
-\sum\limits_{k=1}^s \Gamma ^{k}_{abt}x_{k}, &~~1\leq a,b\leq s<t\leq n,\\
0, &~~1\leq a,b,t\leq s.
\end{split}
\end{cases}
\end{equation*}

If $1\leq a\leq s<b\leq n$, $1\leq c,t\leq n$, then
$\nu_{abt}^{c}=0,\lambda_{abt}^{c}=-\Gamma_{abt}^{c},$ and
\begin{equation*}
[y_a,y_b,x_t]_2=
\begin{cases}\begin{split}
-\sum\limits_{k=1}^s \Gamma ^{k}_{abt}x_{k}, &~~1\leq a,t\leq s<b\leq n,\\
-\sum\limits_{k=s+1}^n \Gamma ^{k}_{abt}x_{k}, &~~1\leq a\leq s<b,t\leq n.
\end{split}
\end{cases}
\end{equation*}
We get \eqref{eq:coadjoint411}.
By Eqs  \eqref{eq:AAA},  \eqref{eq:BBB}, \eqref{eq:bbb}, \eqref{eq:MB21} and \eqref{eq:ppp1}. A  similar discussion to the above,
we get \eqref{eq:coadjoint412} and \eqref{eq:coadjoint413}. We omit  the calculation process. The proof is complete.
\end{proof}

\begin{lemma}\label{lem:RelationJacobi} Let $A$ be a 3-Lie algebra with  an involutive derivation $D$, and  the multiplication of $A$ in the basis $\{ x_1, $ $\cdots, x_s, $ $x_{s+1},$ $ \cdots, x_n \}$ be \eqref{eq:ooo}, where  $x_i\in A_1, x_j\in A_{-1}$, $1\leq i\leq s,$ $s+1\leq j\leq n$.
 Then we have

\begin{equation}\label{eq:Jacobi1}
\begin{cases}\begin{split}
&\sum_{k=s+1}^n\sum_{t=1}^s\big(\Gamma^{k}_{ija}\Gamma^{t}_{kbc}+
\Gamma^{k}_{ijb}\Gamma^{t}_{akc}+\Gamma^{k}_{ijc}\Gamma^{t}_{abk}\big)=0,\\
&\sum_{k=s+1}^n\sum_{t=1}^s\big(\Gamma^{c}_{abk}\Gamma^{k}_{ijt}+
\Gamma^{k}_{ija}\Gamma^{t}_{kbt}+\Gamma^{k}_{ijb}\Gamma^{c}_{akt}\big)=0, \\
&\sum_{k=1}^s\sum_{t=1}^s(\Gamma^{k}_{abi}\Gamma^{c}_{jkt}+
\Gamma^{c}_{jak}\Gamma^{k}_{bit}+
\Gamma^{c}_{jbk}\Gamma^{k}_{iat})=0,\\
&\sum_{t=1}^s(\sum_{k=1}^s(\Gamma^{k}_{abi}\Gamma^{c}_{kjt}+
\Gamma^{k}_{abj}\Gamma^{c}_{ikt})-
\sum_{k=s+1}^n\Gamma^{c}_{abk}\Gamma^{k}_{ijt})=0,\\
&\sum_{t=1}^s(\sum_{k=1}^s(\Gamma^{k}_{cja}\Gamma^{t}_{kbi}+\Gamma^{k}_{cjb}\Gamma^{t}_{aki}-
\Gamma^{k}_{abi}\Gamma^{t}_{cjk})+
\sum_{k=s+1}^n\Gamma^{k}_{cji}\Gamma^{t}_{abk})=0,\\
&\sum_{t=1}^s(\sum_{k=1}^s(\Gamma^{b}_{aik}\Gamma^{k}_{cjt}+
\Gamma^{k}_{cja}\Gamma^{b}_{kit}-\Gamma^{b}_{cjk}\Gamma^{k}_{ait})+
\sum_{k=s+1}^s\Gamma^{k}_{cji}\Gamma^{b}_{akt})=0,\\
&\sum_{t=1}^s(\sum_{k=s+1}^n(\Gamma^{k}_{aij}\Gamma^{b}_{ckt}+\Gamma^{b}_{cak}\Gamma^{k}_{ijt})+
\sum_{k=1}^s(\Gamma^{b}_{cik}\Gamma^{k}_{jat}+
\Gamma^{b}_{cjk}\Gamma^{k}_{ait}))=0,\\
&\sum_{t=1}^s(\sum_{k=1}^s(\Gamma^{k}_{bci}\Gamma^{t}_{akj}+
\Gamma^{k}_{bcj}\Gamma^{t}_{aik})-
\sum_{k=s+1}^n\Gamma^{k}_{aij}\Gamma^{t}_{bck})=0,\\
&\mbox{where} ~ 1\leq a,b,c\leq s, ~~ s+1\leq i,j\leq n,\\
\end{split}\end{cases}
\end{equation}
\begin{equation}\label{eq:Jacobi2}
\begin{cases}\begin{split}
&\sum_{t=s+1}^n\big(\sum_{k=s+1}^n\Gamma^{c}_{abk}\Gamma^{k}_{ijt}+
\sum_{k=1}^s\big(\Gamma^{k}_{ija}\Gamma^{c}_{kbt}+
\Gamma^{k}_{ijb}\Gamma^{c}_{akt}+
\Gamma^{c}_{ijk}\Gamma^{k}_{abt}\big)\big)=0,\\
&\sum_{k=1}^s\sum_{t=s+1}^n\big(\Gamma^{i}_{jak}\Gamma^{k}_{bct}+
\Gamma^{i}_{jbk}\Gamma^{k}_{cat}+
\Gamma^{i}_{jck}\Gamma^{k}_{abt}\big)=0,\\
&\sum_{k=1}^s\sum_{t=s+1}^n(\Gamma^{b}_{ajk}\Gamma^{k}_{ict}+
\Gamma^{k}_{icj}\Gamma^{b}_{akt})+
\sum_{k=s+1}^n\Gamma^{b}_{ick}\sum_{t=1}^s\Gamma^{k}_{ajt}=0,\\
&\sum_{t=s+1}^n\big(\sum_{k=1}^s(\Gamma^{k}_{abj}\Gamma^{i}_{ckt}+\Gamma^{i}_{cjk}\Gamma^{k}_{abt})+
\sum_{k=s+1}^n\big(\Gamma^{i}_{cak}\Gamma^{k}_{bjt}+\Gamma^{i}_{cbk}\Gamma^{k}_{jat}\big)\big)=0,\\
&\mbox{where}~~ 1\leq a,b,c, i\leq s, ~~ s+1\leq j\leq n,
\end{split}\end{cases}
\end{equation}

\begin{equation}\label{eq:Jacobi3}
\begin{cases}\begin{split}
&\sum_{k=1}^s\sum_{t=s+1}^n\big(\Gamma^{k}_{ija}\Gamma^{t}_{kbc}+
\Gamma^{k}_{ijb}\Gamma^{t}_{akc}+\Gamma^{k}_{ijc}\Gamma^{t}_{abk}\big)=0,\\
&\sum_{k=1}^s\sum_{t=s+1}^n\big(\Gamma^{c}_{abk}\Gamma^{k}_{ijt}+
\Gamma^{k}_{ija}\Gamma^{c}_{kbt}+\Gamma^{k}_{ijb}\Gamma^{c}_{akt}\big)=0, \\
&\sum_{k=s+1}^n\sum_{t=s+1}^n(\Gamma^{k}_{ibc}\Gamma^{a}_{jkt}+
\Gamma^{a}_{jbk}\Gamma^{k}_{cit}+
\Gamma^{a}_{jck}\Gamma^{k}_{ibt})=0,\\
&\sum_{t=s+1}^n(\sum_{k=1}^s\Gamma^{k}_{jbi}\Gamma^{t}_{kac}+
\sum_{k=s+1}^n(\Gamma^{k}_{jba}\Gamma^{t}_{ikc}+
\Gamma^{k}_{jbc}\Gamma^{t}_{iak}-
\Gamma^{k}_{iac}\Gamma^{t}_{jbk}))=0,\\
&\sum_{t=s+1}^n(\sum_{k=s+1}^n(\Gamma^{c}_{iak}\Gamma^{k}_{jbt}+\Gamma^{k}_{jba}\Gamma^{c}_{ikt}-
\Gamma^{c}_{jbk}\Gamma^{k}_{iat})-\sum_{k=1}^s\Gamma^{k}_{jbi}\Gamma^{c}_{kat})=0,\\
&\sum_{t=s+1}^n(\sum_{k=s+1}^n(\Gamma^{k}_{abi}\Gamma^{t}_{kjc}+\Gamma^{k}_{abj}\Gamma^{t}_{ikc})-
\sum_{k=1}^s\Gamma^{c}_{abk}\Gamma^{k}_{ijt})=0,\\
&\sum_{t=s+1}^n(\sum_{k=s+1}^n(\Gamma^{k}_{abi}\Gamma^{t}_{kjc}+
\Gamma^{k}_{abj}\Gamma^{t}_{ikc})-
\sum_{k=1}^s\Gamma^{k}_{ijc}\Gamma^{t}_{abk})=0,\\
&\sum_{t=s+1}^n\big(\sum_{k=1}^s(\Gamma^{k}_{ijc}\Gamma^{a}_{bkt}+\Gamma^{a}_{bck}\Gamma^{k}_{ijt})+
\sum_{k=s+1}^n\big(\Gamma^{a}_{bik}\Gamma^{k}_{jct}+\Gamma^{a}_{bjk}\Gamma^{k}_{cit}\big)\big)=0,\\
&\mbox{where} ~~s+1\leq a,b,c\leq n,~~1\leq i,j\leq s,\\
\end{split}\end{cases}
\end{equation}

\begin{equation}\label{eq:Jacobi4}
\begin{cases}\begin{split}
&\sum_{t=1}^s\big(\sum_{k=1}^s\Gamma^{c}_{abk}\Gamma^{k}_{jit}+
\sum_{k=s+1}^n\Gamma^{k}_{jia}\Gamma^{c}_{kbt}+
\sum_{k=s+1}^n\Gamma^{k}_{jib}\Gamma^{c}_{akt}\big)+
\sum_{k=s+1}^n\sum_{t=s+1}^n\Gamma^{c}_{jik}\Gamma^{k}_{abt}=0,\\
&\sum_{k=s+1}^n\sum_{t=1}^s(\Gamma^{i}_{jak}\Gamma^{k}_{bct}+
\Gamma^{i}_{jbk}\Gamma^{k}_{cat}+
\Gamma^{i}_{jck}\Gamma^{k}_{abt})=0,\\
&\sum_{t=1}^s(\sum_{k=s+1}^n(\Gamma^{k}_{jbc}\Gamma^{i}_{akt}+\Gamma^{i}_{ajk}\Gamma^{k}_{bct})+
\sum_{k=1}^s(\Gamma^{i}_{abk}\Gamma^{k}_{cjt}+
\Gamma^{i}_{ack}\Gamma^{k}_{jbt}))=0,\\
&\sum_{t=1}^s(\sum_{k=s+1}^n(\Gamma^{c}_{jbk}\Gamma^{k}_{iat}+
\Gamma^{k}_{iaj}\Gamma^{c}_{kbt})+
\sum_{k=1}^s\Gamma^{c}_{iak}\Gamma^{k}_{jbt})=0,\\
&\mbox{where}~~~1\leq j\leq s, ~~ s+1\leq a,b,c, i\leq n.
\end{split}\end{cases}
\end{equation}
\end{lemma}

\begin{proof} Thanks to \eqref{eq:jacobi1},  \eqref{eq:1} and \eqref{eq:2}, for all  $1\leq a,b,c\leq s,~~$ $ s+1\leq i,j\leq n,$

 $0=\mu(x_{i},x_{j},\mu(x_{a},x_{b},x_{c}))$

 \hspace{2mm} $=\mu(\mu(x_{i},x_{j},x_{a}),x_{b},x_{c})+\mu(x_{a},\mu(x_{i},x_{j},x_{b}),x_{c})
+\mu(x_{a},x_{b},\mu(x_{i},x_{j},x_{c}))$

\vspace{1mm} \hspace{2mm}$=(\sum\limits_{k=s+1}^n\Gamma^{k}_{ija}\sum\limits_{t=1}^s\Gamma^{t}_{kbc}+\sum\limits_{k=s+1}^n\Gamma^{k}_{ijb}\sum\limits_{t=1}^s\Gamma^{t}_{akc}+
\sum\limits_{k=s+1}^n\Gamma^{k}_{ijc}\sum\limits_{t=1}^s\Gamma^{t}_{abk})x_{t}.$

\vspace{2mm} Therefore,
\begin{equation*}
\aligned \
&\sum_{k=s+1}^n\Gamma^{k}_{ija}\sum_{t=1}^s\Gamma^{t}_{kbc}+
&\sum_{k=s+1}^n\Gamma^{k}_{ijb}\sum_{t=1}^s\Gamma^{t}_{akc}+
&\sum_{k=s+1}^n\Gamma^{k}_{ijc}\sum_{t=1}^s\Gamma^{t}_{abk}=0,
\endaligned
\end{equation*}
and

\hspace{2mm}$\mu(x_{i},x_{j},\mu(x_{a},x_{b},x_{c}^{*}))=-\sum\limits_{k=s+1}^n\Gamma^{c}_{abk}\sum\limits_{t=1}^s\Gamma^{k}_{ijt}x_{t}^{*},$

\vspace{2mm}\hspace{2mm} $\mu(\mu(x_{i},x_{j},x_{a}),x_{b},x_{c}^{*})+\mu(x_{a},\mu(x_{i},x_{j},x_{b}),x_{c}^{*})+\mu(x_{a},x_{b},\mu(x_{i},x_{j},x_{c}^{*}))$,\\

\vspace{2mm}  $=(\sum\limits_{k=s+1}^n\Gamma^{k}_{ija}\sum\limits_{t=1}^s\Gamma^{t}_{kbt}+
\sum\limits_{k=s+1}^n\Gamma^{k}_{ijb}\sum\limits_{t=1}^s\Gamma^{c}_{akt})x_{t}^{*}.$
 It follows
 $$\sum_{k=s+1}^n\sum_{t=1}^s\big(\Gamma^{c}_{abk}\Gamma^{k}_{ijt}+
\Gamma^{k}_{ija}\Gamma^{t}_{kbt}+\Gamma^{k}_{ijb}\Gamma^{c}_{akt}\big)=0.$$

By a similar discussion, for others cases of $1\leq a, b, c, i, j\leq n$, we get \eqref{eq:Jacobi2}-\eqref{eq:Jacobi4}.
\end{proof}

\begin{lemma}\label{lem:RelationDeribation1} Let $a\delta^{*}$ and  $a\psi^{*}$ be defined as \eqref{eq:qqq1} and \eqref{eq:ppp1}, respectively. Then $a\delta^{*}$ and  $a\psi^{*}$ satisfy $$a\delta^{*}(B_{1}\wedge B_{1})\subseteq  Der B_{2}~~~ \mbox{ and }~~~a\psi^{*}( B_{2}\wedge B_{2})\subseteq Der B_{1}.$$
 \end{lemma}

\begin{proof} First we prove $a\delta^{*}( B_{1}\wedge B_{1})\subseteq Der B_{2}$.

Thanks to  \eqref{eq:MB21}, and \eqref{eq:coadjoint1}-\eqref{eq:coadjoint413}, we have the following identities

$$
\begin{cases}\begin{split}
&\Delta^{*}(a\delta^{*}_{x_{i}x_{j}}y_{a},y_{b},y_{c})+
  \Delta^{*}(y_{a},a\delta^{*}_{x_{i}x_{j}}y_{b},y_{c})+ \Delta^{*}(y_{a},y_{b},a\delta^{*}_{x_{i}x_{j}}y_{c})=0,\\
   &a\delta^{*}_{x_{i}x_{j}}\Delta^{*}(y_{a},y_{b},y_{c})=0, \\
 &\mbox{where}1\leq a,b,c\leq s, 1\leq i,j\leq s, ~\mbox{ or}~~  1\leq i\leq s<j\leq n;
 \\ &~\mbox{or}~~ 1\leq a,b\leq s <c\leq n, ~~ 1\leq i,j\leq s;
\end{split}
\end{cases}$$

$$
\begin{cases}\begin{split}
&\Delta^{*}(a\delta^{*}_{x_{i}x_{j}}y_{a},y_{b},y_{c})+
\Delta^{*}(y_{a},a\delta^{*}_{x_{i}x_{j}}y_{b},y_{c})+ \Delta^{*}(y_{a},y_{b},a\delta^{*}_{x_{i}x_{j}}y_{c})\\
&=(\sum\limits_{k=s+1}^n\Gamma^{k}_{ija}\sum\limits_{t=1}^s\Gamma^{t}_{kbc}+
\sum\limits_{k=s+1}^n\Gamma^{k}_{ijb}\sum\limits_{t=1}^s\Gamma^{t}_{akc}+
\sum\limits_{k=s+1}^s\Gamma^{k}_{ijc}\sum\limits_{t=1}^s\Gamma^{t}_{abk})y_{t},\\
&a\delta^{*}_{x_{i}x_{j}}\Delta^{*}(y_{a},y_{b},y_{c})=0,\\
&\mbox{where} ~~s+1\leq i, j\leq n,~~1\leq a, b, c\leq s;
\end{split}
\end{cases} $$

$$
\begin{cases}\begin{split}
&
\Delta^{*}(a\delta^{*}_{x_{i}x_{j}}y_{a},y_{b},y_{c})+
\Delta^{*}(y_{a},a\delta^{*}_{x_{i}x_{j}}y_{b},y_{c})+ \Delta^{*}(y_{a},y_{b},a\delta^{*}_{x_{i}x_{j}}y_{c})\\
&=(\sum\limits_{k=s+1}^n\Gamma^{k}_{ija}\sum\limits_{t=s+1}^n\Gamma^{t}_{kbc}+
\sum\limits_{k=s+1}^n\Gamma^{k}_{ijb}\sum\limits_{t=s+1}^n\Gamma^{t}_{akc})y_{t},\\
&a\delta^{*}_{x_{i}x_{j}}\Delta^{*}(y_{a},y_{b},y_{c})=\sum_{k=1}^s\Gamma^{k}_{abc}\sum_{t=s+1}^n\Gamma^{t}_{ijk}y_{t},~\\
& ~\mbox{where} ~~s+1\leq i, j\leq n,~~ 1\leq a, b\leq s<c\leq n;\end{split}
\end{cases}$$

$$
\begin{cases}\begin{split}
&\Delta^{*}(a\delta^{*}_{x_{i}x_{j}}y_{a},y_{b},y_{c})+
  \Delta^{*}(y_{a},a\delta^{*}_{x_{i}x_{j}}y_{b},y_{c})+ \Delta^{*}(y_{a},y_{b},a\delta^{*}_{x_{i}x_{j}}y_{c})\\
&=(-\sum_{k=1}^s\Gamma^{k}_{ija}\sum_{t=1}^s\Gamma^{t}_{kbc}-
\sum_{k=1}^s\Gamma^{k}_{ijb}\sum_{t=1}^s\Gamma^{t}_{akc}-
\sum_{k=s+1}^s\Gamma^{k}_{ijc}\sum_{t=1}^s\Gamma^{t}_{abk})y_{t},\\
&a\delta^{*}_{x_{i}x_{j}}\Delta^{*}(y_{a},y_{b},y_{c})=-\sum_{k=1}^s\Gamma^{k}_{abc}\sum_{t=1}^s\Gamma^{t}_{ijk}y_{t},\\
& \mbox{where} ~~1\leq a,b\leq s<c\leq n, 1\leq i\leq s<j\leq n,\\
&~~ \mbox{or}~~1\leq j\leq s<i\leq n.
\end{split}
\end{cases}$$

By the above discussion and Eqs \eqref{eq:Jacobi1}-\eqref{eq:Jacobi4}, we get
$$a\delta^{*}_{x_{i}x_{j}}\Delta^{*}(y_{a},y_{b},y_{c})=\Delta^{*}(a\delta^{*}_{x_{i}x_{j}}y_{a},y_{b},y_{c})+
  \Delta^{*}(y_{a},a\delta^{*}_{x_{i}x_{j}}y_{b},y_{c})+ \Delta^{*}(y_{a},y_{b},a\delta^{*}_{x_{i}x_{j}}y_{c}),$$
that is,  $a\delta^{*}_{x_{i}x_{j}}\in Der(B_2)$, for all $1\leq i, j\leq n$.

Apply Lemma \ref{lem:B2B1}, Lemma \ref{lem:RelationJacobi} and a similar discussion to the above, we get
 $a\psi^{*}_{x_{i}x_{j}}\in Der(B_1)$, for all $1\leq i, j\leq n$. The proof is complete.
\end{proof}

\begin{theorem} \label{thm:last} Let $(B_1\oplus B_2, [ \cdot,  \cdot, \cdot]_1)$ and  $(B_1\oplus B_2, [ \cdot,  \cdot, \cdot]_2)$ be 3-Lie algebras in Lemma \ref{lem:B1B2} and Lemma \ref{lem:B2B1}, respectively. Then
the 5-tuple $(B_{1}\oplus B_{2}, [\cdot,\cdot,\cdot]_{B_{1}\oplus B_{2}}, ( \cdot , \cdot ), B_{1}, B_{2})$ is a $4n$-dimensional standard Manin triple of 3-Lie algebras, where for $\forall u, v, w\in B_{1}, \alpha, \beta, \xi\in B_{2}$,
\begin{equation}\label{eq:manin}
[u+\alpha, v+\beta, w+\xi]_{B_1\oplus B_2}=[u+\alpha, v+\beta, w+\xi]_1+[u+\alpha, v+\beta, w+\xi]_2,
\end{equation}
and for $\forall x\in B_1, \theta\in B_2$, $(x, \theta)=\langle x, \theta\rangle.$
\end{theorem}

\begin{proof}
By Lamma \ref{lem:RelationDeribation1}, $a\delta^{*}$ and $a\psi^{*}$ satisfy  $a\delta^{*}( B_{1}\wedge B_{1})\subseteq Der B_{2}$ and $a\psi^{*}( B_{2}\wedge B_{2})\subseteq Der B_{1}$, respectively.

Next we only need to prove that  identities  \eqref{eq:JOCABI1}, \eqref{eq:JOCABI2}, \eqref{eq:JOCABI3}  and \eqref{eq:JOCABI4} below hold, since
 \eqref{eq:JOCABI1} is equivalent to \eqref {eq:22'},  \eqref{eq:JOCABI2} is equivalent to \eqref {eq:22''},  \eqref{eq:JOCABI3} is equivalent to \eqref {eq:33'},  and  \eqref{eq:JOCABI4} is equivalent to \eqref {eq:33''} in the 3-algebra
 $(B_{1}\oplus B_{2}, [\cdot,\cdot,\cdot]_{B_{1}\oplus B_{2}})$, respectively, where,  $\forall x_a, x_b, x_c,$ $ y_i, y_j, y_a,$ $ y_b, y_c, x_i,$ $ x_j, x^*_a, x^*_b, x^*_c,$ $ y^*_i, y^*_j,$ $ y^*_a, y^*_b, y^*_c,$ $ x^*_i, x^*_j\in \Pi_1\cup\Pi_2$,

\begin{equation}\label{eq:JOCABI1}
 \begin{array}{llll}
 \left\{\begin{array}{l}
 \mu({x_a, x_b},a\psi^{*}_{y_{i}y_{j}}x_{c})=[a\delta^{*}_{x_{a}x_{b}}y_{i},y_{j},x_{c}]_2+
 [y_{i},a\delta^{*}_{x_{a}x_{b}}y_{j},x_{c}]_2+a\psi^{*}_{y_{i}y_{j}}\mu({x_{a},x_{b}},x_{c}),\\
 \mu({x_a, x_b},a\psi^{*}_{y_{i}y_{j}}x_{c}^{*})=[a\delta^{*}_{x_{a}x_{b}}y_{i},y_{j},x_{c}^{*}]_2+
 [y_{i},a\delta^{*}_{x_{a}x_{b}}y_{j},x_{c}^{*}]_2+a\psi^{*}_{y_{i}y_{j}}\mu({x_{a},x_{b}},x_{c}^{*}),\\
 \mu({x_a, x_b},a\psi^{*}_{y_{i}^{*}y_{j}}x_{c})=[a\delta^{*}_{x_{a}x_{b}}y_{i}^{*},y_{j},x_{c}]_2+
 [y_{i}^{*},a\delta^{*}_{x_{a}x_{b}}y_{j},x_{c}]_2+a\psi^{*}_{y_{i}^{*}y_{j}}\mu({x_{a},x_{b}},x_{c}),\\
  \mu({x_a^{*}, x_b},a\psi^{*}_{y_{i}y_{j}}x_{c})=[a\delta^{*}_{x_{a}x_{b}}y_{i},y_{j},x_{c}]_2+
 [y_{i},a\delta^{*}_{x_{a}x_{b}}y_{j},x_{c}]_2+a\psi^{*}_{y_{i}y_{j}}\mu({x_{a},x_{b}},x_{c}),\\
 \mu({x_a^{*}, x_b},a\psi^{*}_{y_{i}y_{j}}x_{c}^{*})=[a\delta^{*}_{x_{a}x_{b}}y_{i},y_{j},x_{c}^{*}]_2+
 [y_{i},a\delta^{*}_{x_{a}x_{b}}y_{j},x_{c}^{*}]_2+a\psi^{*}_{y_{i}y_{j}}\mu({x_{a},x_{b}},x_{c}^{*}),\\
 \mu({x_a^{*}, x_b},a\psi^{*}_{y_{i}^{*}y_{j}}x_{c})=[a\delta^{*}_{x_{a}x_{b}}y_{i}^{*},y_{j},x_{c}]_2+
 [y_{i}^{*},a\delta^{*}_{x_{a}x_{b}}y_{j},x_{c}]_2+a\psi^{*}_{y_{i}^{*}y_{j}}\mu({x_{a},x_{b}},x_{c}),\\
 \end{array}\right.
\end{array}
 \end{equation}

\begin{equation}\label{eq:JOCABI2}
 \begin{array}{llll}
 \left\{\begin{array}{l}
 \Delta^{*}({y_a, y_b},a\delta^{*}_{x_{i}x_{j}}y_{c})=[a\psi^{*}_{y_{a}y_{b}}x_{i},x_{j},y_{c}]_1+
 [x_{i},a\psi^{*}_{y_{a}y_{b}}x_{j},y_{c}]_1+a\delta^{*}_{x_{i}x_{j}}\Delta^{*}({y_{a},y_{b}},y_{c}),\\
 \Delta^{*}({y_a, y_b},a\delta^{*}_{x_{i}x_{j}}y_{c}^{*})=[a\psi^{*}_{y_{a}y_{b}}x_{i},x_{j},y_{c}^{*}]_1+
[x_{i},a\psi^{*}_{y_{a}y_{b}}x_{j},y_{c}^{*}]_1+a\delta^{*}_{x_{i}x_{j}}\Delta^{*}({y_{a},y_{b}},y_{c}^{*}),\\
 \Delta^{*}({y_a, y_b},a\delta^{*}_{x_{i}^{*}x_{j}}y_{c})=[a\psi^{*}_{y_{a}y_{b}}x_{i}^{*},x_{j},y_{c}]_1+
[x_{i}^{*},a\psi^{*}_{y_{a}y_{b}}x_{j},y_{c}]_1+a\delta^{*}_{x_{i}^{*}x_{j}}\Delta^{*}({y_{a},y_{b}},y_{c}),\\
  \Delta^{*}({y_a^{*}, y_b},a\delta^{*}_{x_{i}x_{j}}y_{c})=[a\psi^{*}_{y_{a}y_{b}}x_{i},x_{j},y_{c}]_1+
[x_{i},a\psi^{*}_{y_{a}y_{b}}x_{j},y_{c}]_1+a\delta^{*}_{x_{i}x_{j}}\Delta^{*}({y_{a},y_{b}},y_{c}),\\
 \Delta^{*}({y_a^{*}, y_b},a\delta^{*}_{x_{i}x_{j}}y_{c}^{*})=[a\psi^{*}_{y_{a}y_{b}}x_{i},x_{j},y_{c}^{*}]_1+
[x_{i},a\psi^{*}_{y_{a}y_{b}}x_{j},y_{c}^{*}]_1+a\delta^{*}_{x_{i}x_{j}}\Delta^{*}({y_{a}y_{b}},y_{c}^{*}),\\
 \Delta^{*}({y_a^{*}, y_b},a\delta^{*}_{x_{i}^{*}x_{j}}y_{c})=[a\psi^{*}_{y_{a}y_{b}}x_{i}^{*},x_{j},y_{c}]_1+
[x_{i}^{*},a\psi^{*}_{y_{a}y_{b}}x_{j},y_{c}]_1+a\delta^{*}_{x_{i}^{*}x_{j}}\Delta^{*}({y_{a},y_{b}},y_{c}),\\
 \end{array}\right.
\end{array}
 \end{equation}

\begin{equation}\label{eq:JOCABI3}
 \begin{array}{llll}
 \left\{\begin{array}{l}
  \mu({x_a, x_b},a\psi^{*}_{y_{i}y_{j}}x_{c})=[a\delta^{*}_{x_{a}x_{b}}y_{i},y_{j},x_{c}]_2-
[y_{i},a\delta^{*}_{x_{c}x_{a}}y_{j},x_{b}]_2-[y_{i},a\delta^{*}_{x_{b}x_{c}}y_{j},x_{a}]_2,\\
  \mu({x_a, x_b},a\psi^{*}_{y_{i}y_{j}}x_{c}^{*})=[a\delta^{*}_{x_{a}x_{b}}y_{i},y_{j},x_{c}^{*}]_2-
 [y_{i},a\delta^{*}_{x_{c}^{*}x_{a}}y_{j},x_{b}]_2-[y_{i},a\delta^{*}_{x_{b}x_{c}^{*}}y_{j},x_{a}]_2, \\
  \mu({x_a, x_b},a\psi^{*}_{y_{i}^{*}y_{j}}x_{c})=[a\delta^{*}_{x_{a}x_{b}}y_{i}^{*},y_{j},x_{c}]_2-
 [y_{i}^{*},a\delta^{*}_{x_{c}x_{a}}y_{j},x_{b}]_2-[y_{i}^{*},a\delta^{*}_{x_{b}x_{c}}y_{j},x_{a}]_2,\\
  \mu({x_a^{*}, x_b},a\psi^{*}_{y_{i}y_{j}}x_{c})=[a\delta^{*}_{x_{a}^{*}x_{b}}y_{i},y_{j},x_{c}]_2-
 [y_{i},a\delta^{*}_{x_{c}x_{a}^{*}}y_{j},x_{b}]_2-[y_{i},a\delta^{*}_{x_{b}x_{c}}y_{j},x_{a}^{*}]_2,\\
  \mu(x_a^{*}, x_b,a\psi^{*}_{y_{i}y_{j}}x_{c}^{*})=[a\delta^{*}_{x_a^*x_{b}}y_{i},y_{j},x^*_c]_2-
 [y_{i},a\delta^{*}_{x_{c}^{*}x_{a}^{*}}y_{j},x_{b}]_2-[y_{i},a\delta^{*}_{x_{b}x_{c}^{*}}y_{j},x_{a}^{*}]_2,
 \\
  \mu({x_a^{*}, x_b},a\psi^{*}_{y_{i}^{*}y_{j}}x_{c})=[a\delta^{*}_{x_{a}^{*}x_{b}}y_{i}^{*},y_{j},x_{c}]_2-
[y_{i}^{*},a\delta^{*}_{x_{c}x_{a}^{*}}y_{j},x_{b}]_2-[y_{i}^{*},a\delta^{*}_{x_{b}x_{c}}y_{j},x_{a}^{*}]_2,\\

 \end{array}\right.
\end{array}
 \end{equation}

\begin{equation}\label{eq:JOCABI4}
 \begin{array}{llll}
 \left\{\begin{array}{l}
 \Delta^{*}({y_a, y_b},a\delta^{*}_{x_{i}x_{j}}y_{c})=[a\psi^{*}_{y_{a}y_{b}}x_{i},x_{j},y_{c}]_1-
 [x_{i},a\psi^{*}_{y_{c}y_{a}}x_{j},y_{b}]_1-[x_{i},a\psi^{*}_{y_{b}y_{c}}x_{j},y_{a}]_1,\\
  \Delta^{*}({y_a, y_b},a\delta^{*}_{x_{i}x_{j}}y_{c}^{*})=[a\psi^{*}_{y_{a}y_{b}}x_{i},x_{j},y_{c}^{*}]_1-
 [x_{i},a\psi^{*}_{y_{c}^{*}y_{a}}x_{j},y_{b}]_1-[x_{i},a\psi^{*}_{y_{b}y_{c}^{*}}x_{j},y_{a}]_1,\\
 \Delta^{*}({y_a, y_b},a\delta^{*}_{x_{i}^{*}x_{j}}y_{c})=[a\psi^{*}_{y_{a}y_{b}}x_{i}^{*},x_{j},y_{c}]_1-
 [x_{i}^{*},a\psi^{*}_{y_{c}y_{a}}x_{j},y_{b}]_1-[x_{i}^{*},a\psi^{*}_{y_{b}y_{c}}x_{j},y_{a}]_1,\\
\Delta^{*}({y_a^{*}, y_b},a\delta^{*}_{x_{i}x_{j}}y_{c})=[a\psi^{*}_{y_{a}^{*},y_{b}}x_{i},x_{j},y_{c}]_1-
 [x_{i},a\psi^{*}_{y_{c}y_{a}^{*}}x_{j},y_{b}]_1-[x_{i},a\psi^{*}_{y_{b}y_{c}}x_{j},y_{a}^{*}]_1,\\
  \Delta^{*}({y_a^{*}, y_b},a\delta^{*}_{x_{i}x_{j}}y_{c}^{*})=[a\psi^{*}_{y_{a}^{*}y_{b}}x_{i},x_{j},y_{c}^{*}]_1-
 [x_{i},a\psi^{*}_{y_{c}^{*}y_{a}^{*}}x_{j},y_{b}]_1-[x_{i},a\psi^{*}_{y_{b}y_{c}^{*}}x_{j},y_{a}^{*}]_1,\\
 \Delta^{*}({y_a^{*}, y_b},a\delta^{*}_{x_{i}^{*}x_{j}}y_{c})=[a\psi^{*}_{y_{a}^{*}y_{b}}x_{i}^{*},x_{j},y_{c}]_1-
 [x_{i}^{*},a\psi^{*}_{y_{c}y_{a}^{*}}x_{j},y_{b}]_1-[x_{i}^{*},a\psi^{*}_{y_{b}y_{c}}x_{j},y_{a}^{*}]_1.\\
 \end{array}\right.
\end{array}
 \end{equation}

First we discuss  Eq \eqref{eq:JOCABI1}. Thanks to  \eqref{eq:coadjoint1} in  Lemma \ref {lem:B1B2}, \eqref{eq:coadjoint411}-\eqref{eq:coadjoint413} in Lemma \ref{lem:B2B1}, and \eqref{eq:1}, we obtain
$$
[a\delta^{*}_{x_{a}x_{b}}y_{i},y_{j},x_{c}]_2+
 [y_{i},a\delta^{*}_{x_{a}x_{b}}y_{j},x_{c}]_2+a\psi^{*}_{y_{i}y_{j}}\mu({x_{a},x_{b}},x_{c})=0, ~\mbox{and}$$
$$\mu({x_a, x_b},a\psi^{*}_{y_{i}y_{j}}x_{c})=0, ~\mbox{where} ~ 1\leq a,b,c\leq s,
  1\leq i,j\leq s, ~\mbox{or} ~1\leq i\leq s<j\leq n.
$$

$$
[a\delta^{*}_{x_{a}x_{b}}y_{i},y_{j},x_{c}]_2+
 [y_{i},a\delta^{*}_{x_{a}x_{b}}y_{j},x_{c}]_2+a\psi^{*}_{y_{i}y_{j}}\mu({x_{a},x_{b}},x_{c})=0, ~\mbox{and}$$
$$\mu({x_a, x_b},a\psi^{*}_{y_{i}y_{j}}x_{c})=0, ~\mbox{where} ~ s+1\leq a,b,c\leq n,
  s+1\leq i,j\leq n, ~\mbox{or} ~1\leq i\leq s<j\leq n.
$$

$$[a\delta^{*}_{x_{a}x_{b}}y_{i},y_{j},x_{c}]_2+
[y_{i},a\delta^{*}_{x_{a}x_{b}}y_{j},x_{c}]_2+a\psi^{*}_{y_{i}y_{j}}\mu({x_{a},x_{b}},x_{c})$$
$$=-\sum_{k=s+1}^n\sum_{t=s+1}^n(\Gamma^{k}_{abi}\Gamma^{t}_{kjc}+\Gamma^{k}_{abj}\Gamma^{t}_{ikc})x_{t},~\mbox{and}$$
$$\mu({x_a, x_b},a\psi^{*}_{y_{i}y_{j}}x_{c})=-\sum_{k=1}^s\sum_{t=s+1}^n\Gamma^{k}_{ijc}\Gamma^{t}_{abk}x_{t}, ~\mbox{where} ~ 1\leq a,b,c\leq s, ~~s+1\leq i,j\leq n.$$

$$[a\delta^{*}_{x_{a}x_{b}}y_{i},y_{j},x_{c}]_2+
 [y_{i},a\delta^{*}_{x_{a}x_{b}}y_{j},x_{c}]_2+a\psi^{*}_{y_{i}y_{j}}\mu({x_{a},x_{b}},x_{c})$$
 $$=-\sum_{k=1}^s\sum_{t=1}^s(\Gamma^{k}_{abi}\Gamma^{t}_{kjc}+\Gamma^{k}_{abj}\Gamma^{t}_{ikc})x_{t}, ~\mbox{and}$$
$$\mu({x_a, x_b},a\psi^{*}_{y_{i}y_{j}}x_{c})=-\sum_{k=s+1}^n\sum_{t=1}^s\Gamma^{k}_{ijc}\Gamma^{t}_{abk}x_{t}, ~\mbox{where} ~ s+1\leq a,b,c\leq n, ~~ 1\leq i,j\leq s+1.
$$

Thanks to Eqs \eqref{eq:Jacobi1}  and
\eqref{eq:Jacobi3},
we get
$$\sum_{k=1}^s\sum_{t=s+1}^n\Gamma^{k}_{ijc}\Gamma^{t}_{abk}x_{t}= \sum_{k=s+1}^n\sum_{t=s+1}^n(\Gamma^{k}_{abi}\Gamma^{t}_{kjc}+\Gamma^{k}_{abj}\Gamma^{t}_{ikc})x_{t}, $$
 for ~~ $1\leq a,b,c\leq s,~~ s+1\leq i, j\leq n;$
and
$$\sum_{k=s+1}^n\sum_{t=1}^s\Gamma^{k}_{ijc}\Gamma^{t}_{abk}x_{t}=\sum_{k=1}^s\sum_{t=1}^s(\Gamma^{k}_{abj}\Gamma^{t}_{ikc}+\Gamma^{k}_{abc}\Gamma^{t}_{ijk})x_{t},$$
  for ~~$s+1\leq a,b,c\leq n,~~ 1\leq i, j\leq s.$

Therefore,  for $1\leq a,b,c\leq s$, or $s+1\leq a,b,c\leq n,$
$$ \mu({x_a, x_b},a\psi^{*}_{y_{i}y_{j}}x_{c})=[a\delta^{*}_{x_{a}x_{b}}y_{i},y_{j},x_{c}]_2+
 [y_{i},a\delta^{*}_{x_{a}x_{b}}y_{j},x_{c}]_2+a\psi^{*}_{y_{i}y_{j}}\mu({x_{a},x_{b}},x_{c}), 1\leq i, j\leq n.$$

By a similar discussion to the above, we have
$$[a\delta^{*}_{x_{a}x_{b}}y_{i},y_{j},x_{c}]_2+
 [y_{i},a\delta^{*}_{x_{a}x_{b}}y_{j},x_{c}]_2+a\psi^{*}_{y_{i}y_{j}}\mu({x_{a},x_{b}},x_{c})=0, ~~\mbox{and}~$$
$$\mu({x_a, x_b},a\psi^{*}_{y_{i}y_{j}}x_{c})=0,$$
where~ $1\leq a,b\leq s<c\leq n, 1\leq i,j\leq s
~\mbox{or}~~1\leq c\leq s<a,b\leq n, ~~ s+1\leq i,j\leq n.
$
$$
[a\delta^{*}_{x_{a}x_{b}}y_{i},y_{j},x_{c}]_2+
 [y_{i},a\delta^{*}_{x_{a}x_{b}}y_{j},x_{c}]_2+a\psi^{*}_{y_{i}y_{j}}\mu({x_{a},x_{b}},x_{c})
$$
$$=-\sum_{k=1}^s\sum_{t=s+1}^n(\Gamma^{k}_{abi}\Gamma^{t}_{kjc}+\Gamma^{k}_{abj}\Gamma^{t}_{ikc}
 +\Gamma^{k}_{abc}\Gamma^{t}_{ijk})x_{t}, ~~\mbox{and}~$$
 $$\mu({x_a, x_b},a\psi^{*}_{y_{i}y_{j}}x_{c})=0,$$
 ~~where~ $1\leq a,b\leq s<c\leq n, s+1\leq i,j\leq n.$

$$[a\delta^{*}_{x_{a}x_{b}}y_{i},y_{j},x_{c}]_2+
 [y_{i},a\delta^{*}_{x_{a}x_{b}}y_{j},x_{c}]_2+a\psi^{*}_{y_{i}y_{j}}\mu({x_{a},x_{b}},x_{c})$$
 $$=-\sum_{k=1}^s\sum_{t=1}^s(\Gamma^{k}_{abj}\Gamma^{t}_{ikc}+\Gamma^{k}_{abc}\Gamma^{t}_{ijk})x_{t},~~\mbox{and}~$$
$$\mu({x_a, x_b},a\psi^{*}_{y_{i}y_{j}}x_{c})=-\sum_{k=s+1}^n\sum_{t=1}^s\Gamma^{k}_{ijc}\Gamma^{t}_{abk}x_{t},$$
 ~~where~ $1\leq a,b\leq s<c\leq n, 1\leq i\leq s<j\leq n.$
$$[a\delta^{*}_{x_{a}x_{b}}y_{i},y_{j},x_{c}]_2+
 [y_{i},a\delta^{*}_{x_{a}x_{b}}y_{j},x_{c}]_2+a\psi^{*}_{y_{i}y_{j}}\mu({x_{a},x_{b}},x_{c})
=-\sum_{k=1}^s\sum_{t=1}^s(\Gamma^{k}_{abi}\Gamma^{t}_{kjc}+\Gamma^{k}_{abc}\Gamma^{t}_{ijk})x_{t},$$
 $$\mu({x_a, x_b},a\psi^{*}_{y_{i}y_{j}}x_{c})=-\sum_{k=s+1}^n\sum_{t=1}^s\Gamma^{k}_{ijc}\Gamma^{t}_{abk}x_{t},$$
 ~~where~$ 1\leq a,b\leq s<c\leq n, 1\leq j\leq s<i\leq n.$

$$[a\delta^{*}_{x_{a}x_{b}}y_{i},y_{j},x_{c}]_2+
 [y_{i},a\delta^{*}_{x_{a}x_{b}}y_{j},x_{c}]_2+a\psi^{*}_{y_{i}y_{j}}\mu({x_{a},x_{b}},x_{c})$$
$$=-\sum_{k=s+1}^n\sum_{t=1}^s(\Gamma^{k}_{abi}\Gamma^{t}_{kjc}+\Gamma^{k}_{abj}\Gamma^{t}_{ikc}
 +\Gamma^{k}_{abc}\Gamma^{t}_{ijk})x_{t},~\mbox{and}~$$
 $$\mu({x_a, x_b},a\psi^{*}_{y_{i}y_{j}}x_{c})=0,~~where~ 1\leq c\leq s<a,b\leq n, 1\leq i,j\leq s.$$

$$[a\delta^{*}_{x_{a}x_{b}}y_{i},y_{j},x_{c}]_2+
 [y_{i},a\delta^{*}_{x_{a}x_{b}}y_{j},x_{c}]_2+a\psi^{*}_{y_{i}y_{j}}\mu({x_{a},x_{b}},x_{c})$$
$$=-\sum_{k=s+1}^n\sum_{t=s+1}^n(\Gamma^{k}_{abi}\Gamma^{t}_{kjc}+\Gamma^{k}_{abc}\Gamma^{t}_{ijk})x_{t},~\mbox{and}~$$
 $$\mu({x_a, x_b},a\psi^{*}_{y_{i}y_{j}}x_{c})=-\sum_{k=1}^s\sum_{t=s+1}^n\Gamma^{k}_{ijc}\Gamma^{t}_{abk}x_{t},
 ~~where~ 1\leq c\leq s<a,b\leq n, 1\leq i\leq s<j\leq n.$$

$$[a\delta^{*}_{x_{a}x_{b}}y_{i},y_{j},x_{c}]_2+
 [y_{i},a\delta^{*}_{x_{a}x_{b}}y_{j},x_{c}]_2+a\psi^{*}_{y_{i}y_{j}}\mu({x_{a},x_{b}},x_{c})$$
$$=-\sum_{k=s+1}^n\sum_{t=s+1}^n(\Gamma^{k}_{abj}\Gamma^{t}_{ikc}+\Gamma^{k}_{abc}\Gamma^{t}_{ijk})x_{t}, ~\mbox{and}~$$
 $$\mu({x_a, x_b},a\psi^{*}_{y_{i}y_{j}}x_{c})=-\sum_{k=1}^s\sum_{t=s+1}^n\Gamma^{k}_{ijc}\Gamma^{t}_{abk}x_{t},$$
 ~~where~ $1\leq c\leq s<a,b\leq n, 1\leq j\leq s<i\leq n.$

Thanks to \eqref{eq:Jacobi1} (in Lemma \ref {lem:RelationJacobi}), for all $1\leq a,b\leq s<c\leq n,$ we have

$$\sum_{k=s+1}^n\sum_{t=1}^s(\Gamma^{k}_{abi}\Gamma^{t}_{kjc}+\Gamma^{k}_{abj}\Gamma^{t}_{ikc}
 +\Gamma^{k}_{abc}\Gamma^{t}_{ijk})x_{t}=0,~s+1\leq i,j\leq n;$$

$$\sum_{k=s+1}^n\sum_{t=1}^s\Gamma^{k}_{ijc}\Gamma^{t}_{abk}x_{t}=
\sum_{k=1}^s\sum_{t=1}^s(\Gamma^{k}_{abj}\Gamma^{t}_{ikc}+\Gamma^{k}_{abc}\Gamma^{t}_{ijk})x_{t},
~1\leq i\leq s<j\leq n;$$

 $$\sum_{k=s+1}^n\sum_{t=1}^s\Gamma^{k}_{ijc}\Gamma^{t}_{abk}x_{t}=
 \sum_{k=1}^s\sum_{t=1}^s(\Gamma^{k}_{abi}\Gamma^{t}_{kjc}+\Gamma^{k}_{abc}\Gamma^{t}_{ijk})x_{t},
 ~, 1\leq j\leq s<i\leq n.$$

By \eqref{eq:Jacobi3} (in Lemma \ref {lem:RelationJacobi}), for $1\leq c\leq s<a,b\leq n$, we have $$\sum_{k=1}^s\sum_{t=s+1}^n(\Gamma^{k}_{abi}\Gamma^{t}_{kjc}+\Gamma^{k}_{abj}\Gamma^{t}_{ikc}
 +\Gamma^{k}_{abc}\Gamma^{t}_{ijk})x_{t}=0,~1\leq i,j\leq s;$$

$$\sum_{k=1}^s\sum_{t=s+1}^n\Gamma^{k}_{ijc}\Gamma^{t}_{abk}x_{t}=
\sum_{k=s+1}^n\sum_{t=s+1}^n(\Gamma^{k}_{abi}\Gamma^{t}_{kjc}+\Gamma^{k}_{abc}\Gamma^{t}_{ijk})x_{t},
~1\leq i\leq s<j\leq n;$$

$$\sum_{k=1}^s\sum_{t=s+1}^n\Gamma^{k}_{ijc}\Gamma^{t}_{abk}x_{t}=
\sum_{k=s+1}^n\sum_{t=s+1}^n(\Gamma^{k}_{abj}\Gamma^{t}_{ikc}+\Gamma^{k}_{abc}\Gamma^{t}_{ijk})x_{t},
~ 1\leq j\leq s<i\leq n.$$

Therefore,  for $ 1\leq i, j\leq n,$ $1\leq a,b\leq s<c\leq n,$ or $1\leq c\leq s<a,b\leq n,$ we have
$$ \mu({x_a, x_b},a\psi^{*}_{y_{i}y_{j}}x_{c})=[a\delta^{*}_{x_{a}x_{b}}y_{i},y_{j},x_{c}]_2+
 [y_{i},a\delta^{*}_{x_{a}x_{b}}y_{j},x_{c}]_2+a\psi^{*}_{y_{i}y_{j}}\mu({x_{a},x_{b}},x_{c}).$$

Summarizing the above discussion, we get that for $1\leq a,b,c\leq n,~1\leq i, j\leq n,$
$$\mu({x_a, x_b},a\psi^{*}_{y_{i}y_{j}}x_{c})=[a\delta^{*}_{x_{a}x_{b}}y_{i},y_{j},x_{c}]_2+
 [y_{i},a\delta^{*}_{x_{a}x_{b}}y_{j},x_{c}]_2+a\psi^{*}_{y_{i}y_{j}}\mu({x_{a},x_{b}},x_{c}),$$
 this is, the first identity in \eqref{eq:JOCABI1} holds. By a similar discussion to the above, we get \eqref{eq:JOCABI2}-\eqref{eq:JOCABI4}.

Thanks to  Theorem \ref{thm:AoplusA'},  $(B_{1}\oplus B_{2}, [\cdot,\cdot,\cdot]_{B_{1}\oplus B_{2}}, ( \cdot , \cdot ), B_{1}, B_{2})$ is a standard Manin triple of 3-Lie algebras. The proof is complete.
\end{proof}

\begin{coro}
 Let $(B_1\oplus B_2, [\cdot,\cdot,\cdot]_1)$ and  $(B_1\oplus B_2, [\cdot,\cdot,\cdot]_2)$ be 3-Lie algebras in Lemma \ref{lem:B1B2} and Lemma \ref{lem:B2B1}, respectively.
Then $((B_1\oplus B_2, [\cdot,\cdot,\cdot]_1),(B_1\oplus B_2, [\cdot,\cdot,\cdot]_2),a\delta^{*},a\psi^{*})$ is an $4n$-dimensional  matched pair.
\end{coro}

\begin{proof}
 Apply Proposition 4.7 in \cite{BCM5} and Theorem \ref{thm:last}.
\end{proof}

 At last of the paper, we construct a sixteen dimensional  Manin triple of 3-Lie algebras by an involutive derivation.

\begin{exam}\label{exam:1}
Let $A$ be a 4-dimensional 3-Lie algebra with  $\dim A^1=2$, and the multiplication of $A$ in the basis $\{x_{1},x_{2},x_{3},x_{4}\}$ be as follows
$$
[x_1, x_3, x_4]=x_2, ~~ [x_1, x_3, x_4]=x_1.
$$
Then the linear mapping $D: A \rightarrow A$ defined by $D(x_i)=x_i$ for $1\leq i\leq 3$ and $D(x_4)=-x_4$ is an involutive derivation of $A$, and
satisfies $x_{1},x_{2},x_{3}\in A_{1}$ and $x_{4}\in A_{-1}.$

By Theorem \ref{thm:last}, $(B_{1}\oplus B_{2}, [\cdot,\cdot,\cdot]_{B_{1}\oplus B_{2}}, ( \cdot , \cdot ), B_{1}, B_{2})$ is  a sixteen dimensional  Manin triple of 3-Lie algebras in the basis $\{x_{1}, \cdots, x_{16}\}$, where $B_1=\langle x_1, \cdots, x_8\rangle$, $B_2=\langle x_9, \cdots, x_{16}\rangle$, and the multiplication $[\cdot,\cdot,\cdot]_{B_{1}\oplus B_{2}}$ is defined as \eqref{eq:manin}.

For convenience, let $B=B_1\oplus B_2$, and $[\cdot, \cdot, \cdot]_B$ the multiplication $[\cdot, \cdot, \cdot]_{B_1\oplus B_2}$.
Then  the multiplication of the Manin triple of  3-Lie algebras in the basis $\{x_{1}, \cdots, x_{16}\}$ is as follows:

$$
[x_{2},x_{3},x_{4}]_B=x_{1},~~ ~[x_{1},x_{3},x_{4}]_B=x_{2},~~~
~[x_{1},x_{4},x_{6}]_B=x_{7},~~~ [x_{2},x_{5},x_{3}]_B=x_{8},$$
$$ [x_{2},x_{4},x_{5}]_B=x_{7}, ~~[x_{3},x_{5},x_{4}]_B=x_{6}, ~~[x_{3},x_{6},x_{4}]_B=x_{5},
~~[x_{3},x_{4},x_{9}]_B=x_{10}, $$
$$[x_{3},x_{1},x_{14}]_B=x_{16},
[x_{3},x_{2},x_{13}]_B=x_{16},
[x_{6},x_{3},x_{9}]_B=x_{16}, ~~[x_{3},x_{6},x_{12}]_B=x_{13}, $$
$$[x_{4},x_{6},x_{9}]_B=x_{15},
[x_{6},x_{4},x_{11}]_B=x_{13},
[x_{2},x_{3},x_{12}]_B=x_{10}, [x_{4},x_{1},x_{11}]_B=x_{10}, $$
$$[x_{2},x_{3},x_{12}]_B=x_{9},
[x_{4},x_{2},x_{11}]_B=x_{9},[x_{1},x_{14},x_{4}]_B=x_{15},[x_{2},x_{4},x_{13}]_B=x_{15},$$
$$[x_{4},x_{3},x_{13}]_B=x_{14},
[x_{4},x_{3},x_{14}]_B=x_{13},
[x_{1},x_{5},x_{11}]_B=x_{16}, [x_{5},x_{2},x_{12}]_B=x_{15},$$
$$[x_{5},x_{3},x_{10}]_B=x_{16},
[x_{3},x_{5},x_{12}]_B=x_{14},
[x_{4},x_{5},x_{10}]_B=x_{15}, [x_{5},x_{4},x_{11}]_B=x_{14},$$
$$[x_{1},x_{6},x_{11}]_B=x_{16},
[x_{6},x_{1},x_{12}]_B=x_{15},
[x_{9},x_{4},x_{11}]_B=x_{2}, [x_{9},x_{12},x_{3}]_B=x_{2},$$
$$[x_{10},x_{4},x_{11}]_B=x_{1},
~~[x_{10},x_{12},x_{3}]_B=x_{1},
[x_{11},x_{1},x_{12}]_B=x_{2}, [x_{11},x_{10},x_{12}]_B=x_{1},~ $$
$$[x_{10},x_{11},x_{5}]_B=x_{8},
[x_{10},x_{5},x_{12}]_B=x_{7},
[x_{11},x_{12},x_{5}]_B=x_{6}, [x_{9},x_{11},x_{6}]_B=x_{8},$$
$$[x_{9},x_{6},x_{12}]_B=x_{7},
[x_{11},x_{12},x_{6}]_B=x_{5},
[x_{13},x_{10},x_{3}]_B=x_{8}, [x_{13},x_{4},x_{10}]_B=x_{7}, $$
$$[x_{13},x_{2},x_{11}]_B=x_{8},
[x_{13},x_{11},x_{4}]_B=x_{6}, [x_{13},x_{12},x_{2}]_B=x_{7},[x_{13},x_{3},x_{12}]_B=x_{6},$$
$$[x_{14},x_{9},x_{3}]_B=x_{8},
[x_{14},x_{4},x_{9}]_B=x_{7}, [x_{14},x_{1},x_{11}]_B=x_{8}, [x_{14},x_{11},x_{4}]_B=x_{5}, $$
$$[x_{14},x_{12},x_{1}]_B=x_{7},
~~[x_{14},x_{3},x_{12}]_B=x_{5},
[x_{10},x_{12},x_{11}]_B=x_{9}, [x_{9},x_{12},x_{11}]_B=x_{10},$$
$$[x_{9},x_{11},x_{14}]_B=x_{16}, ~[x_{9},x_{14},x_{12}]_B=x_{15},
[x_{10},x_{11},x_{13}]_B=x_{16}, [x_{10},x_{13},x_{12}]_B=x_{15},$$
$$  ~~[x_{1},x_{6},x_{3}]_B=x_{8}, ~~[x_{3},x_{4},x_{10}]_B=x_{9}, [x_{11},x_{12},x_{13}]_B=x_{14}, [x_{11},x_{14},x_{12}]_B=x_{13}.
$$
\end{exam}

\begin{theorem} \label{thm:last} The 16-dimensional 3-Lie algebra $B$ in Example \ref{exam:1} is 2-solvable but non-nilpotent, and  $B$ has the smallest ideal
 $$I=\langle  x_1, x_2, x_7, x_8, x_9, x_{10},  x_{15}, x_{16}\rangle,$$
 which $I$ is an abelian ideal, and $\dim B^1=12$.
 \end{theorem}

\begin{proof}From the multiplication, $$I=\langle  x_1, x_2, x_7, x_8, x_9, x_{10},  x_{15}, x_{16}\rangle$$
is the smallest ideal of $B$, and satisfies $[I, I, B]_B=0.$  Therefore, $I$ is an abelian ideal.

It is clear that $B$ is a 16-dimensional 3-Lie algebra with $\dim B^1=12$, and
$$B^1=[B, B, B]_B=\langle x_1, x_2, x_5, x_6, x_7, x_8, x_9, x_{10},  x_{13}, x_{14}, x_{15}, x_{16}\rangle.$$

 Since for any positive integer $r$, $B^r=B^1$, $B$ is non-nilpotent.
  From
$$B^{(2)}=[B^1, B^1, B]_B=I, ~~~ B^{(3)}=[B^{(2)}, B^{(2)}, B ]_B=[I, I, B]_B=0,$$
$B$ is a 2-solvable
3-Lie algebra.
\end{proof}

\subsection*{Acknowledgements}
Ruipu Bai was supported by the Natural
Science Foundation of Hebei Province (A2018201126).

\end{document}